\documentclass[10pt,a4paper]{amsart}
\usepackage{amsmath,amssymb, amsbsy}
\usepackage{color,psfrag}
\usepackage[dvips]{graphicx}
\usepackage{color}
\theoremstyle{plain}

\newtheorem{theorem}{Theorem}[section]
\newtheorem{proposition}{Proposition}[section]
\newtheorem{lemma}{Lemma}[section]

\newtheorem{remark}{Remark}[section]
\newtheorem{definition}{Definition}[section]

\newcommand{\eps}{\varepsilon}

\newcommand{\authorfootnotes}{\renewcommand\thefootnote{\@fnsymbol\c@footnote}}%

\newcommand{\norm}[1]{\left\Vert#1\right\Vert}

\newcommand{\be} {\begin{equation}}
\newcommand{\ee} {\end{equation}}
\newcommand{\bea} {\begin{eqnarray}}
\newcommand{\eea} {\end{eqnarray}}
\newcommand{\Bea} {\begin{eqnarray*}}
\newcommand{\Eea} {\end{eqnarray*}}
\newcommand{\pa} {\partial}

\newcommand{\ba} {\beta}

\newcommand{\na}{\nabla}
\newcommand{\ga} {\gamma}
\newcommand{\Ga} {\Gamma}

\newcommand{\De}{\Delta}
\newcommand{\la} {\lambda}
\newcommand{\si} {\sigma}

\newcommand{\no} {\nonumber}

\newcommand{\lab} {\label}
\newcommand{\va} {\varphi}
\newcommand{\f}{\frac}
\newcommand{\R}{\mathbb R}

\newcommand{\Rn}{\mathbb R^N}

\catcode`\@=11

\numberwithin{equation}{section} \allowdisplaybreaks

\usepackage[text={6in,8.6in},centering]{geometry}
\begin{document}
    
\title[Fractional Hardy equations]{Fractional Hardy equations with critical and supercritical exponents}

\author{Mousomi Bhakta, Debdip Ganguly and Luigi Montoro}
\address{M. Bhakta, Department of Mathematics, Indian Institute of Science Education and Research, Dr. Homi Bhaba Road, Pune-411008, India}
\email{mousomi@iiserpune.ac.in}
\address{D. Ganguly, Department of Mathematics, Indian Institute of Technology Delhi, IIT Campus, Hauz Khas, New Delhi, Delhi 110016, India}
\email{debdip@maths.iitd.ac.in}
\address{L. Montoro, Dipartimento di Matematica, UNICAL, Ponte Pietro Bucci 31 B, 87036 Arcavacata
di Rende, Cosenza, Italy.}
\email{montoro@mat.unical.it}

\subjclass[2010]{35B07; 35B65; 35R09; 35S05; 47G30}
\keywords{Super-critical exponent, fractional Laplacian, Hardy's inequality, Harnack inequality, moving plane method.}
\date{}

\maketitle

\begin{abstract}
We study the existence/nonexistence and qualitative properties of the positive solutions to the problem
\begin{equation}
  \tag{$\mathcal P$}
\left\{\begin{aligned}
      (-\De)^s u -\theta\frac{u}{|x|^{2s}}&=u^p - u^q \quad\text{in }\,\, \Rn \\
      u &> 0 \quad\text{in }\,\, \Rn\no\\
            u &\in \dot{H}^s(\Rn)\cap L^{q+1}(\Rn), \no
 \end{aligned}
  \right.
\end{equation}
where $s\in (0,1)$, $N>2s$, $q>p\geq{(N+2s)}/{(N-2s)}$, $\theta\in(0, \Lambda_{N,s})$ and $\Lambda_{N,s}$ is the sharp constant in the fractional Hardy inequality.  For  qualitative properties of the solutions we mean, both the radial symmetry that is  obtained by using the moving plane method 
in a nonlocal setting on the whole $\mathbb{R}^N$, and   upper bound  behavior
of the solutions. To this last end  we use a representation result
that allows us to transform the original problem into a new nonlocal problem in a weighted fractional space.
\end{abstract}


\section{Introduction}
In this paper we consider the following singular problem
\begin{equation}
  \tag{$\mathcal P$}\label{eq:a3}
\left\{\begin{aligned}
      (-\De)^s u -\theta\frac{u}{|x|^{2s}}&=u^p - u^q \quad\text{in }\,\, \Rn\\
      u & > 0 \quad\text{in }\,\, \Rn\\
            u &\in \dot{H}^s(\Rn)\cap L^{q+1}(\Rn), 
 \end{aligned}
  \right.
\end{equation}
where $s\in(0,1)$ is fixed, $(-\De)^s$ denotes the  fractional Laplace operator which can be defined for the Schwartz class functions $\mathcal{S}(\Rn)$  as follows
\begin{equation} \label{De-u}
  \left(-\Delta\right)^su(x): = a_{N,s} \text{P.V.} \int_{\Rn}\frac{u(x)-u(y)}{|x-y|^{N+2s}} \, dy, \quad a_{N,s}= \frac{4^{s}\Ga(N/2+ s)}{\pi^{N/2}|\Ga(-s)|},
\end{equation}
$N>2s$,  $q>p\geq 2^*_s-1=\f{N+2s}{N-2s}$  and $\theta\in(0, \Lambda_{N,s})$, where $\Lambda_{N,s}$ is the sharp constant in the Hardy inequality 
$$\Lambda_{N,s} \, \int_{\Rn} \, \frac{|u(x)|^2}{|x|^{2s}} \, dx \leq \int_{\Rn}|\xi|^{2s}|\mathcal{F}(u)(\xi)|^2 \, d\xi.$$ In the above inequality $\mathcal{F}(u)$ denotes the Fourier transform of $u$ and 
\be\lab{La-ns}\Lambda_{N,s}=2^{2s}\frac{\Gamma^2(\frac{N+2s}{4})}{\Gamma^2(\frac{N-2s}{4})}.\ee Moreover, $$\lim_{s\to 1}\Lambda_{N,s}\to \bigg(\frac{N-2}{2}\bigg)^2.$$
\begin{definition}
Let $s\in(0,1)$. We define the homogeneous fractional Sobolev space of order $s$
as $$\dot{H}^s(\Rn):=\bigg\{u\in L^{2^*_s}(\Rn) \,:\, \int_{\Rn}|\xi|^{2s}|\mathcal{F}(u)|^2 \, d\xi<\infty  \bigg\},$$
namely the completion of $C^\infty_0(\Rn)$ under the norm 
\be\lab{dot-H}\|u\|_{\dot{H}^s(\Rn)}^2:=\int_{\Rn}|\xi|^{2s}|\mathcal{F}(u)|^2 \, d\xi.\ee
\end{definition}
Using Plancherel's identity, for $s\in (0,1)$, $N\geq 1$, we obtain an equivalent expression of the norm \eqref{dot-H}, namely 
\begin{equation}\lab{norm-equi}
\int_{\Rn}|\xi|^{2s}|\mathcal{F}(u)|^2 \, d\xi=\frac{a_{N,s}}{2}\int_{\Rn}\int_{\Rn}\frac{|u(x)-u(y)|^2}{|x-y|^{N+2s}} \, dx \, dy, \quad \forall \, u\in \dot{H}^s(\Rn),\end{equation}
see \cite{FLS-1}.
Using the above norm-equivalence and a density argument, it follows that
\be\lab{Hardy}
\Lambda_{N,s}\int_{\Rn}\frac{u^2(x)}{|x|^{2s}} \, dx\leq \frac{a_{N,s}}{2}\int_{\Rn}\int_{\Rn}\frac{|u(x)-u(y)|^2}{|x-y|^{N+2s}} \, dx \, dy, \quad \forall \, u\in \dot{H}^s(\Rn),
\ee
see \cite[Remark 1.3]{DMPS}.

The notion of solutions to \eqref{eq:a3} that we consider in this paper is given in the following 
\begin{definition}
We say $u\in\dot{H}^s(\Rn)\cap L^{q+1}(\Rn)$ is a weak solution of \eqref{eq:a3} if for every $\va\in \dot{H}^s(\Rn)\cap L^{q+1}(\Rn)$, we have
\begin{equation}\label{1.1bis}
\frac{a_{N,s}}{2}\int_{\Rn}\int_{\Rn}\frac{(u(x)-u(y))(\va(x)-\va(y))}{|x-y|^{N+2s}}\,dx\, dy-\theta\int_{\Rn}\frac{u\va}{|x|^{2s}}dx=\int_{\Rn}(u^p-u^q)\va \, dx.
\end{equation}
\end{definition}
We note that in the above definition, 
$$\int_{\Rn}u^p\va \, dx\leq \bigg(\int_{\Rn}u^{p+1}\, dx\bigg)^\frac{p}{p+1}\bigg(\int_{\Rn}\va^{p+1}\, dx\bigg)^\frac{1}{p+1}<\infty,$$
since $2^*_s\leq p+1<q+1$ implies $u, \va\in L^{p+1}(\Rn)$ by interpolation inequality. 
\begin{remark}\lab{r:1}
Note that, $0<\theta<\Lambda_{N,s}$. Therefore using  \eqref{dot-H}, \eqref{norm-equi} and \eqref{Hardy} we see that 
$$\bigg(\frac{a_{N,s}}{2}\displaystyle\int_{\Rn}\int_{\Rn}\frac{|u(x)-u(y)|^2}{|x-y|^{N+2s}} \, dx \, dy\bigg)^\frac{1}{2}$$ is an equivalent norm to
$$\bigg(\frac{a_{N,s}}{2}\int_{\Rn}\int_{\Rn}\frac{|u(x)-u(y)|^2}{|x-y|^{N+2s}}dx\,dy\, - \,  \theta \, \int_{\Rn} \frac{|u|^2}{|x|^{2s}}dx \bigg)^\frac{1}{2}.$$
\end{remark}
Very recently, a great deal of attention is given to the mathematical study of the following class of semilinear elliptic problems
\begin{equation}\lab{5-1-1}
\left\{\begin{aligned}
      (-\De)^s u -\theta\frac{u}{|x|^{2s}}&=f(x,u) \quad\text{in }\,\, \Rn\\
      u &\gneqq 0 \quad\text{in } \, \,\Rn\\
            u &\in \dot{H}^s(\Rn), 
            \end{aligned}
  \right.
\end{equation}

where $f(x,u)$ is a superlinear function in $u$. The local case. i.e, when $s = 1$ and $f(x, u) = u^\frac{N+2}{N-2}$ has been 
thoroughly investigated  by Terracini in \cite{ST}. Existence, uniqueness and qualitative properties of the solutions
have been shown by the author.   It is natural to study these type of problems in the nonlocal framework. The study 
of nonlocal Yamabe problem in the Euclidean space, i.e., for the case when $\theta = 0$ and $f(x,u)=u^{2^*_s-1}$ in \eqref{5-1-1}
has been studied in \cite{CLO}. The authors  completely characterize the solutions using the moving plane technique.
However for $\theta \neq 0,$  the presence of Hardy potential requires a new understanding to deal with the problem.  

\medskip 
In the last few years there are several  
 works related to Eq.\eqref{5-1-1}, see for instance (\cite{BBGM, BCP, DMPS, FF1, FF} and references therein). {Authors in \cite{BBGM} have studied the two sided Green function estimate of fractional Hardy operator and the integral representation of solution.} Dipierro, et. al in  \cite{DMPS} have considered $f(x,u)=u^{2^*_s-1}$ and proved existence, qualitative properties  and asymptotic behaviour of solutions both at $0$ and infinity. In \cite{BCP}, Bhakta, et. al have studied \eqref{5-1-1} with $f(x,u)=K(x){|u|^{2^*(t)-2}u}/{|x|^t}+f(x)$, where $0\leq t<2s$, $2^*_s(t):={2(N-t)}/{N-2s}$, $K$ is a positive
continuous function on $\mathbb{R}^{N}$, with $K(0)=1=\lim_{|x|\to\infty}K(x)$ and $f$ is a nonnegative nontrivial functional in the dual space of $\dot{H}^s(\Rn)$.  They have established existence of at least two positive solutions when $K\geq 1$ and $\|f\|_{(\dot{H}^s)'}$ is small enough but $f\not\equiv 0$. On the other hand, Fall and Felli in \cite{FF} have studied the more general relativistic Schr\"{o}dinger type equations.
We also quote the papers \cite{BM-2, BM, BMS-2} for  nonlocal equations (without Hardy potential) where the nonlinearity involved is of the form $f(x,u)=u^p-u^q$. In the regular case, $\theta=0$, \eqref{eq:a3} is reduced to the problem without Hardy term and it was studied in \cite{BM, BMS-2}. 

\

In this paper we are interested in studying the existence/nonexistence 
of solutions to problem \eqref{eq:a3} and we further investigate qualitative properties like radial symmetry and upper bound of solutions.  

This kind of critical and supercritical exponent problem has been extensively studied when $\theta =0$ and $s = 1$ (local case) by 
Merle and Peletier in \cite{MP1}. When $\theta \neq 0$ and  $s=1$ (local case), a complete classification of the nature of singularities and of the asymptotic 
behaviour of solutions near zero was studied in \cite{BS}.
In the nonlocal case, $s\in(0,1)$, the fractional framework introduces nontrivial difficulties that have interest in itself. Moreover due to nonlocal structure 
of the equation, none of the methods developed in the local case  $s= 1$   can be used for  \eqref{eq:a3}. 
There are many new novelties that need to be exploited to tackle the problem  \eqref{eq:a3} which are describe below.

\medskip


Below we state the main results of this paper. The first one is an existence and non-existence result, stated in the following 
\begin{theorem}\label{thm-non-exist}
Let $0 < \theta < \Lambda_{N, s}$. Then the following holds 

\medskip

\begin{itemize}
\item [$(i)$] If $q>p=2^{*}_s-1$, then \eqref{eq:a3} does not admit any nontrivial solution.  

\

\item [$(ii)$] If $q>p>2^{*}_s-1$, then there exists a positive solution to \eqref{eq:a3}.
\end{itemize}
\end{theorem} 
In order to prove the first part of Theorem \ref{thm-non-exist} we  show a Pohozaev identity (see Proposition \ref{pro:Pohozaev})  using the representation of fractional laplacian $(-\Delta)^s$ as harmonic extension (we give the details in Section \ref{section-2}). To prove the second part of our result, we deal with a constrained minimization problem (see formula \ref{eq:peach})  and then we use Lagrange multipliers technique to get a solution to \eqref{eq:a3}. Moreover in Section  \ref{section-2} we show that, in some case,  weak solutions to \eqref{eq:a3} are indeed strong solutions in $\mathbb R^n \setminus \{0\}$, see Proposition \ref{l:8}.

In the  next two results (see also  Remark \ref{rem:giurie} below) we deduce  qualitative properties of solutions to \eqref{eq:a3} when  $$q+1>(p-1)\frac{N}{2s}.$$  

In particular the first one is related to  the radial symmetry  of the solutions. We have the following 
\begin{theorem}\label{thm:rad}Let $p>2^{*}_s-1$ and 
\begin{equation}\label{eq:qmaggioredip}
q+1>(p-1)\frac{N}{2s}.
\end{equation} 
Assume, $0< \theta< \Lambda_{N,s}$ and let $u$ be a positive   solution to  \eqref{eq:a3}. Then $u$ is radial and radially decreasing with respect to the origin. Namely there exists some strictly decreasing function $$w:(0,+\infty)\rightarrow (0,+\infty)$$ such that
$$u(x)=w(r), \quad r=|x|.$$
\end{theorem}
We prove the result exploiting the moving plane method, well understood in the case of  local problems. In the non local case we refer to
\cite{BMS, CLO, FeWa,  JW2, JW3, MPS2, soave}. However the presence of the Hardy potential in equation \eqref{eq:a3},  on the one hand, makes difficult to use the technique developed in~\cite{CLO} where the authors exploited 
the equivalence of \eqref{eq:a3} to an appropriate integral equation and on the other hand,  solutions to \eqref{eq:a3} loose (because the presence of Hardy potential) regularity at the origin. For this reason, in order to prove the radial symmetry 
of every solution to \eqref{eq:a3}, we  use an  approach based on 
the moving plane method  for weak solutions of the equation in all $\mathbb{R}^N$  taking care  of all difficulties introduced by the presence of Hardy potential.

In the second one, we establish an estimate from above of the asymptotic behavior of the solutions to~\eqref{eq:a3}. We have the following 
\begin{theorem}\lab{t:3}
Let $p>2^{*}_s-1$ and
\begin{equation}\nonumber
q+1>(p-1)\frac{N}{2s}.
\end{equation} 
Assume, $0< \theta< \Lambda_{N,s}$ and let $u$ be a positive   solution to  \eqref{eq:a3}. Then there exists a positive constant $C=C(p,q,s,N,\gamma_{\theta},\|u\|_{L^{q+1}(\mathbb R^N)})$  such that
$$0 < u(x)\leq C|x|^{-\ga_{\theta}}, \quad |x|>0$$
where $\gamma_{\theta}\in (0, (N-2s)/2) $ is a  parameter   determined as the unique solution to \eqref{psi-theta-hairagione}.
\end{theorem}
This result is achieved by a delicate use of Moser iteration technique for an equivalent problem to~\eqref{eq:a3} (see in particular problem \eqref{eq:a3'}) in a weighted fractional space. However due to the presence of super critical exponent and the unbounded domain, the standard Moser 
iteration technique seems to be insufficient to tackle the problem.  There are many intriguing steps involved in the Moser iteration to prove the 
upper bound of solutions, see Proposition \ref{l:5}. 

\

The rest of the paper is organized as follows. In Section 2 we provide a regularity result (i.e. Proposition \ref{l:8})  and  we briefly discuss the different representation of fractional Laplacian using harmonic extension method. Then  we prove  Theorem \ref{thm-non-exist}. In Section 3 we prove Theorem \ref{thm:rad} and Theorem \ref{t:3}.

\

\noindent {\bf Notation.} Throughout the present paper, we denote by $C, C_1, C_2, C', \tilde C, \hat{C}, \cdots$ positive constants that may
vary from line to line. If necessary, the dependence of these constants will be made precise. By $f\lesssim g$ we mean $f\leq Cg$ and by $f\approx g$  we mean $cg\leq f\leq Cg$, for some positive constants $c, C$.
\section{Existence,  nonexistence and regularity results for solutions to \eqref{eq:a3}}\label{section-2}
We start providing a useful result about the regularity of the solutions to the problem \eqref{eq:a3}. We have the following 
\begin{proposition}\lab{l:8}
Let $u$ be a nonnegative solution of \eqref{eq:a3} with 
$$q+1>(p-1)\frac{N}{2s}.$$ 
Then  $u>0$ in $\Rn$ and $u\in C^\infty(\Rn\setminus\{0\})$.
\end{proposition}
\begin{proof}
Let $x_0\in\Rn\setminus\{0\}$. Since $|x|^{-2s}$ is bounded away from $0$ and $q+1>(p-1)\frac{N}{2s}$, using the Moser iteration method as in Theorem \cite[Theorem 1.3]{BM}, it can be shown that $u\in L^\infty(B_r(x_0))$, for some $r>0$. Since, $x_0$ is arbitrary, it implies that $u\in L^\infty_{loc}(\Rn\setminus\{0\})$. 

{\sc Claim}: \be\lab{int-con}\displaystyle\int_{\Rn}\frac{|u(x)|}{(1+|x|)^{N+2s}} \, dx\, <\infty.\ee

We prove this claim in two steps. 

\

{\bf Step 1}: In this step we show that 
$$\int_{B_1(0)}\frac{|u(x)|}{(1+|x|)^{N+2s}} \, dx <\infty.$$

Indeed, by H\"older and Hardy inequalities we get
\begin{eqnarray*}
&&\int_{B_1(0)}\frac{|u(x)|}{(1+|x|)^{N+2s}} \, dx \, \leq \int_{B_1(0)}\frac{|u(x)|}{|x|^{s}(1+|x|)^{N+2s}} \, dx \\\nonumber &&\leq  C\bigg(\int_{B_1(0)}\frac{|u(x)|^2}{|x|^{2s}} \, dx \, \bigg)^\frac{1}{2}\leq C \, \bigg(\int_{\Rn} \, \frac{|u(x)|^2}{|x|^{2s}} \, dx \, \bigg)^\frac{1}{2} \\\nonumber&&\leq C\|u\|_{\dot{H}^s(\Rn)}<\infty.
\end{eqnarray*}

\

{\bf Step 2}: In this step we show that  
$$\int_{\Rn\setminus B_1(0)}\frac{|u(x)|}{(1+|x|)^{N+2s}} \, dx<\infty.$$

For this first we note that
\bea\lab{4-1-5}
\int_{\Rn\setminus B_1(0)}\frac{|u(x)|}{(1+|x|)^{N+2s}} \, dx \, &\leq& \bigg(\int_{\Rn}\frac{|u(x)|^2}{(1+|x|)^{N+2s}} \, 
dx \, \bigg)^\frac{1}{2}\bigg(\int_{\Rn\setminus B_1(0)}\frac{dx}{(1+|x|)^{N+2s}}\bigg)^\frac{1}{2}\no\\
&\leq& C\bigg(\int_{\Rn}\frac{|u(x)|^2}{(1+|x|)^{N+2s}} \, dx \, \bigg)^\frac{1}{2}.
\eea
We also observe that since $u\in\dot{H}^s(\Rn)$, then 
$$\displaystyle\int_{\Rn}\int_{\Rn}\frac{|u(x)-u(y)|^2}{|x-y|^{N+2s}} \, dx \, dy<\infty$$
and Fubini's theorem implies that
\be\lab{23-1-1}\int_{\Rn}\frac{|u(x)-u(y)|^2}{(1+|x-y|)^{N+2s}} \, dx \, \leq \int_{\Rn}\frac{|u(x)-u(y)|^2}{|x-y|^{N+2s}} \, dx \, < \, \infty, \quad \text{a.e. } y\in\Rn.\ee
Fix some $y \in \Rn$ for which the above expression holds and let $R\in \mathbb R$ be such that $R\geq 10|y|$.
For $|x|\geq R$, we deduce that 
\begin{equation}\label{eq:cartpar}\frac 1C(1+|x|)^{N+2s}\leq(1+|x-y|)^{N+2s}\leq C(1+|x|)^{N+2s},
\end{equation} for some constant $C>1$.  Moreover
\be\lab{23-1-2}\displaystyle\int_{\Rn}\frac{|u(y)|^2}{(1+|x|)^{N+2s}} \, dx \, <\infty.\ee Therefore, using the inequality $$|u(x)|^2\leq 2\big[|u(x)-u(y)|^2+|u(y)|^2\big],$$ the fact that by Sobolev inequality
$$\|u\|_{L^2(B_R(0))}< +\infty,$$ and  equations \eqref{23-1-1}, \eqref{eq:cartpar}, \eqref{23-1-2}, it follows that
\begin{eqnarray}\label{eq:chicchi}
&&\int_{\Rn}\frac{|u(x)|^2}{(1+|x|)^{N+2s}} \, dx \,= \int_{B_R(0)}\frac{|u(x)|^2}{(1+|x|)^{N+2s}} \, dx \, + \int_{\mathbb R^N \setminus B_R(0)}\frac{|u(x)|^2}{(1+|x|)^{N+2s}} \, dx \,\\\nonumber
&&\leq  \int_{B_R(0)}\frac{|u(x)|^2}{(1+|x|)^{N+2s}} \, dx \, + 2\int_{\mathbb R^N \setminus B_R(0)}\frac{|u(x)-u(y)|^2}{(1+|x|)^{N+2s}} \, dx \,+  2\int_{\mathbb R^N \setminus B_R(0)}\frac{|u(y)|^2}{(1+|x|)^{N+2s}} \, dx \,\\\nonumber
&&\leq  \int_{B_R(0)}\frac{|u(x)|^2}{(1+|x|)^{N+2s}} \, dx \,+ C\int_{\mathbb R^N}\frac{|u(x)-u(y)|^2}{(1+|x-y|)^{N+2s}} \, dx \,+C \int_{\mathbb R^N }\frac{|u(y)|^2}{(1+|x|)^{N+2s}} \, dx \,<\infty.
\end{eqnarray}
Substituting \eqref{eq:chicchi} in  \eqref{4-1-5} completes Step 2. 

\

Combining Step 1 and Step 2, we  prove $$\int_{\Rn}\frac{|u(x)|}{(1+|x|)^{N+2s}} \, dx<\infty.$$ Hence the claim follows.

Now using the Schauder estimate in the spirit of \cite[Corollary 2.4 and 2.5]{RS1}, it follows that $u\in C^{\ba+2s}\big(B_{\frac{r}{2}}(x_0)\big)$, for some $\ba>0$. As a consequence, for any $x\in \mathbb R^N\setminus \{0\}$ we can write (see problem \eqref{eq:a3}) in  a pointwise way that  
$$(-\De)^s u(x) -\theta\frac{u(x)}{|x|^{2s}}=u^p(x) - u^q(x)$$ 
and consequently  we deduce that $u>0$ in $\Rn\setminus\{0\}$. Indeed suppose, that is not true. That means there exists $y_0\in\Rn\setminus\{0\}$ such that $u(y_0)=0$, which in turn implies that $u$ attains it's minimum at $y_0$. Since $(-\De)^su$ is defined in pointwise sense in $\Rn\setminus\{0\}$, from the integral representation \eqref{De-u} of $(-\De)^s$,  we see that $(-\De)^su(y_0)<0$. On the other hand $$\theta \, \frac{u(y_0)}{|y_0|^{2s}} \, + \, u^p(y_0)\, -\, u^q(y_0)\, =\, 0$$ and this leads to the contradiction. 

Finally a 
bootstrap procedure using Schauder estimate in the spirit of \cite[Corollary 2.4]{RS1}, yields $u\in C^\infty(\Rn\setminus\{0\})$.
\end{proof}
\begin{remark}\label{rem:giurie}
Proposition \ref{l:8} indeed  contains a strong maximum principle for \eqref{eq:a3}. Actually all non negative solutions to \eqref{eq:a3} are positive providing $q+1>(p-1){N}/{2s}$.
\end{remark}
Moreover we have the following
\begin{lemma}\lab{l:decay}
Let $u$ be a solution of \eqref{eq:a3} with 
$$q+1>(p-1)\frac{N}{2s}.$$ 
Then $u\to 0$ as $|x|\to\infty$.
\end{lemma}
\begin{proof}
Thanks to Sobolev inequality, we already have $u\in L^{2^*_s}(\Rn)$. Therefore, in order to prove the lemma, it is enough to show that $u\in C^\sigma(\{|x|>1\})$ uniformly, for some $\si>0$.  To prove this, first we claim the following:

\medskip

{\sc Claim}: There exists a constant $C>0$ and an uniform radius $r>0$ such that $$\sup_{B_r(x)} u\leq C\quad \forall\, x \quad \text{with}\quad |x|>1,$$
where $C$ and $r$ do not depend on $x$.

Indeed the claim follows arguing along the same line as in the proof of \cite[Theorem1.3]{BM}. A careful look on their proof would reveal that 
the constant $C$ ($L^\infty$ bound of the solution) depends only on $p,q,s,N,\|u\|_{L^{q+1}(\Rn)}$ and the radius of the support of the cut-off function that was chosen
 in the proof \cite[Theorem1.3]{BM}, more importantly it is independent of the choice of the point $x.$ Therefore we can choose the radius of the support of the cut-off function there as ${1}/{4},$ which in turn yields $r={1}/{8}$. Hence the claim follows.

In particular, the above claim implies 
\begin{equation}\label{eq:vascovita}
\sup_{|x|>1} u<C,
\end{equation} 
where $C$ is as above. Therefore from \eqref{eq:a3}, we also have $$\sup_{|x|>1} (-\De)^s u<C_1.$$ Combining this with \eqref{int-con}, it follows from the Schauder estimate \cite[Corollary 2.5]{RS1} that 
$$\|u\|_{C^{\sigma}(B_{\frac{r}{2}}(x_0))}\leq C_2,$$ for some $\si\in(0,2s)$,  where $C_2$ does not depend on $x_0$ and $r$ is obtained as in the above claim. Since  $\sup_{|x|>1} u<C,$ then it follows that $u$ is uniformly continuous for $|x|>1$. Hence the lemma follows.
\end{proof}
In the following, we recall now  the other useful representation of fractional laplacian $(-\De)^s$ as harmonic extension \cite{CS}, which we will use to establish the  Pohozaev identity, Proposition~\ref{pro:Pohozaev}. Using the harmonic extension method,  the fractional laplacian $(-\De)^s$ can be seen as a trace class operator (see \cite{CS},  \cite[Appendix A]{FF}).
Let $u \in \dot{H}^s(\Rn)$ be a solution of \eqref{eq:a3} and  define $w:=E_{s}(u)\in \dot H^1(\mathbb R^{N+1}_+; y^{1-2s})$  its $s$-harmonic extension to the upper half space 
$$\R^{N+1}_+:=\big\{z=(x, y)\,\,: \,\,  x\in \mathbb R^N, \, y>0 \big\},$$
that is the unique   solution, see \cite[Proposition 6.2]{FF}, to the following problem
\begin{align}\lab{A-41}
 \begin{cases}
  \mbox{div}(y^{1-2s} \na w)=0  &\quad\mbox{in}\,\, \R^{N+1}_+\\
  w=u &\quad\mbox{on}\,\,\mathbb{R}^N\times \{y=0\},
 \end{cases}
\end{align}
where the space 
$\dot H^1(\mathbb R^{N+1}_+; y^{1-2s})$ is defined as the completion  of $C_0^\infty(\overline{\R_+^{N+1}})$ with respect to the  norm 
$$\norm{\varphi}_{\dot H^1(\mathbb R^{N+1}_+; y^{1-2s})}:=\bigg(\int_{\R_+^{N+1}}y^{1-2s}|\bigtriangledown \varphi|^2 \, dy \, dx\bigg)^\frac{1}{2}.$$
In addition  
\begin{equation}\label{11-12-20-1}-\lim_{y \to 0^+}y^{1-2s}\frac{\pa w}{\pa y}(x,y)=k_{s} (-\De)^{s}u(x)\quad \mbox{in}\,\, \dot H^{-s} (\mathbb R^N),
\end{equation}
where  $\dot H^{-s} (\mathbb R^N)$ denotes the dual of $\dot H^{s}  (\mathbb R^N)$ and 
$$k_{s}=\frac{2^{1-2s}\Ga(1-s)}{\Ga(s)}.$$ 
Moreover, the extension operator $E_s(u):\dot{H}^s(\Rn)\to \dot H^1(\mathbb R^{N+1}_+; y^{1-2s})$ is an isometry, that is, for any $u\in\dot{H}^s(\Rn)$, we have 
\begin{equation}\label{11-12-20-3}
\|E_s(u)\|_{H^1(\mathbb R^{N+1}_+;\, y^{1-2s})}=\sqrt{k_s}\|u\|_{\dot{H}^s(\Rn)},
\end{equation}
see \cite{GS}.
Conversely, for any  $w \in \dot H^1(\mathbb R^{N+1}_+; y^{1-2s})$, we denote its trace on $\Rn\times\{y=0\}$ as $Tr(w):=w(.,0)$. By 
 \cite[Proposition 6.2]{FF} this trace operator is well defined and belongs to $\dot{H}^{s}(\Rn)$, namely there exists a (unique) linear trace operator
 $$T: \dot H^1(\mathbb R^{N+1}_+; y^{1-2s})\rightarrow \dot{H}^{s}(\Rn)$$ 
 such that 
$T(w)(x, y):=w(x,0)$ for any $w\in C^{\infty}_c(\overline{\mathbb R^{N+1}_+})$. Moreover  this trace operator satisfies
\begin{align}\lab{tr-ineq}
k_s \|T(w)\|^2_{\dot{H}^{s}(\Rn)} \leq \|w\|_{\dot H^1(\mathbb R^{N+1}_+; y^{1-2s})}^2,
\end{align}
for all $w\in\dot H^1(\mathbb R^{N+1}_+; y^{1-2s})$.
Consequently,
\be\lab{tr-ineq1}
S_{s,N}\bigg(\int_{\Rn}|T(w)|^{2^*_s}dx\bigg)^{\frac{2}{2^*_s}} \leq \int_{\R^{N+1}_+}y^{1-2s}|\na w|^2 \, dy \, dx,
\ee
for all $w\in\dot H^1(\mathbb R^{N+1}_+; y^{1-2s})$. Inequality \eqref{tr-ineq1} is called the trace inequality.
With the above representation,  $\eqref{A-41}$ can be rewritten as:
\begin{equation}\lab{A-42}
\left\{\begin{aligned}
      \text{div}(y^{1-2s} \na w) &=0  \quad\text{in}\quad \R^{N+1}_+,\\
-\lim_{y \to 0^+}y^{1-2s}\frac{\pa w}{\pa y}(x,y) &= k_s\bigg[\theta\frac{w(.,0)}{|.|^{2s}}+w^p(.,0)-w^{q}(.,0)\bigg] \quad\text{on}\quad\mathbb{R}^N \times \{ y \,= \, 0 \}.
          \end{aligned}
  \right.
\end{equation}

A function $w\in\dot H^1(\mathbb R^{N+1}_+; y^{1-2s})$ with $Tr(w)\in L^{q+1}(\Rn)$, is said to be a weak solution to \eqref{A-42} if for all $\va \in \dot H^1(\mathbb R^{N+1}_+; y^{1-2s})$ with $Tr(\va)\in L^{q+1}(\Rn)$, 
we have
\begin{eqnarray}\lab{A-43}
 \int_{\R^{N+1}_+}y^{1-2s}\na w \na \,  \va\ dx \, dy  &=& \bigg[\theta\int_{\Rn}\frac{w(x,0)}{|x|^{2s}} \,  \va(x,0)\, dx  +\int_{\Rn}w^p(x,0) \, \va (x,0)\, dx \no\\
&&\qquad\qquad- \int_{\Rn}w^{q}(x,0) \, \va(x,0)\, dx\bigg].
\end{eqnarray}
Note that for any weak solution $w \in\dot H^1(\mathbb R^{N+1}_+; y^{1-2s})$ to \eqref{A-42}, the function 
$u:=\mbox{Tr}(w)=w(.,0)\in \dot{H}^{s}(\Rn)$ is a weak solution to \eqref{eq:a3}. 

We prove Theorem \ref{thm-non-exist} using Pohozaev identity and constrained minimisation method. Before proving the theorem we need to  prove a key proposition. 
\begin{proposition}\label{pro:Pohozaev}
Let  $q > 2_s^* - 1$ and $u$ be a solution of 
\begin{equation} \label{prop-nonexist}
\left\{\begin{aligned}
      (-\De)^s u -\theta\frac{u}{|x|^{2s}}&=u^{2_s^* - 1} - u^q \quad\text{in }\quad \Rn, \\
      u &> 0 \quad\text{in }\quad \Rn,\\
            u &\in \dot{H}^s(\Rn)\cap L^{q+1}(\Rn).
 \end{aligned}
  \right.
  \end{equation}
Then there holds the following identity
\begin{equation} \label{P-identity}
\frac{N-2s}{2} \int_{\Rn}\left(u^{2^*_s}-u^{q+1}\right)\, dx= N\int_{\Rn}\left(\frac{1}{2^*_s}u^{2^*_s}-\frac{1}{q+1}u^{q+1}\right)\, dx.
\end{equation}
\end{proposition}
\begin{proof} 
{We prove this proposition using the  harmonic extension method described in this section.  Let us define 
\begin{equation}\label{eq:eva1}
f(u)=f(x,u):=\theta \,  \frac{u}{|x|^{2s}} + u^{2_s^* - 1} - u^q
\end{equation}
and 
\begin{equation}\label{eq:eva2}
F(u)=F(x,u):= \displaystyle\int_{0}^{u}\, f(s) \, ds=\frac{\theta}{2} \,  \frac{u^2}{|x|^{2s}} + \frac{1}{2_s^*}u^{2_s^*} -\frac {1}{q+1}u^{q+1}.
\end{equation}
We point out that by Hardy inequality \eqref{Hardy} and by \eqref{prop-nonexist} it follows  $F(u) \in L^{1}(\Rn).$  Suppose,
  $w$ is the harmonic extension of $u.$ Then $w$ is a solution of

  \begin{equation} \label{upper-half}
\left\{\begin{aligned}
      \mbox{div}(y^{1-2s} \nabla w) = 0     \quad\text{in }\quad \mathbb{R}^{N+1}_{+}, \\
      -\lim_{y \to 0^+}y^{1-2s}\frac{\pa w}{\pa y}(x,y) =k_s f(w(., 0)) \quad\text{on }\quad \Rn \times \{ 0 \},\\
 \end{aligned}
  \right.
  \end{equation}
where $f$ is defined in \eqref{eq:eva1}. 
We denote a point in $\R^{N+1}_+$ as $z = (x, y) \in \Rn \times \mathbb{R}^{+}.$ Let us denote the following sets, for $\varepsilon > 0,\, R> 0 \ \mbox{and}\  \rho > 0$: 
\begin{align*}\label{set1}
& \mathcal{O}_{\varepsilon, R, \rho} := \{ z= (x, y) \in \Rn\times [\varepsilon, + \infty) : \rho^2 < |z|^2 < R^2 \};  \\
& \partial \mathcal{O}^{1}_{\varepsilon, R, \rho} := \{ z= (x, y) \in \Rn \times \{ y = \varepsilon \} : \rho^2 - \varepsilon^2 < |x|^2< R^2 - \varepsilon^2 \};  \\
&  \partial \mathcal{O}^{2}_{\varepsilon, R} := \{ z= (x, y) \in \Rn \times (\varepsilon, \infty)  : |z|^{2} = R^2 \}; \\
&  \partial \mathcal{O}^{3}_{\varepsilon, \rho} := \{ z= (x, y) \in \Rn \times (\varepsilon, \infty)  : |z|^{2} = \rho^2 \}.
\end{align*}
The unit (outward) normal 
 on $\partial \mathcal{O}_{\varepsilon, R}$ is given by 
 \begin{equation}\label{unitnormal}
 \vec{\eta}
 =
\left\{\begin{aligned}
(0, 0, \ldots, 0, -1) &  \quad \mbox{on} \quad  \partial \mathcal{O}^1_{\varepsilon, R, \rho}, \\
 \frac{z}{R}  &  \quad \mbox{on} \quad    \partial \mathcal{O}^{2}_{\varepsilon, R},\\
        -\frac{z}{\rho}  &  \quad \mbox{on} \quad    \partial \mathcal{O}^{3}_{\varepsilon, \rho}.\\
     \end{aligned}
  \right.
  \end{equation}
 On $\mathcal{O}_{\varepsilon, R, \rho},$ there holds 
 $$
 \mbox{div} (y^{1-2s} \nabla w) (z, \nabla w)= \mbox{div} \bigg[y^{1-2s} \nabla w (z, \nabla w) - y^{1-2s} z \frac{|\nabla w|^2}{2}\bigg] + \frac{N-2s}{2} y^{1-2s} |\nabla w|^2.
 $$
 Taking $(z, \nabla w)$ as test function and   integrating by parts on $\mathcal{O}_{\varepsilon, R, \rho}$  we obtain 
 \begin{align}\label{eq:vasc0}
 0=\int_{\mathcal{O}_{\varepsilon, R, \rho}} \mbox{div} (y^{1-2s} \nabla w) (z, \nabla w) \, dx \, dy 
 & =  \int_{\mathcal{O}_{\varepsilon, R, \rho}}  \mbox{div} \left[y^{1-2s} \nabla w (z, \nabla w) - y^{1-2s} z \frac{|\nabla w|^2}{2} \right] \, dx \, dy \\\nonumber
& \quad +    \frac{N-2s}{2} \int_{\mathcal{O}_{\varepsilon, R, \rho}} y^{1-2s} |\nabla w|^2  \, dx \, dy \\\nonumber
&= \underbrace{ \int_{\partial \mathcal{O}^{1}_{\varepsilon, R, \rho}} y^{1-2s} \left[ (x, \nabla_{x} w) (-\partial_{y} w)-y|\partial_y w|^2 + \frac{y}{2} |\nabla w|^2 \right] \, {\rm d}\sigma}_{:= I_{\varepsilon, R, \rho}^{1}} \\\nonumber
&+ \underbrace{\int_{\partial \mathcal{O}^{2}_{\varepsilon, R}} y^{1-2s} \left[ \frac{1}{R} |(z, \nabla w)|^2 - \frac{R}{2} |\nabla w|^2 \right] \, {\rm d}\sigma}_{:= I_{\varepsilon, R}^{2} } \\\nonumber
&+ \underbrace{ \int_{\partial \mathcal{O}^{3}_{\varepsilon, \rho}} y^{1-2s} \left[ -\frac{1}{\rho} |(z, \nabla w)|^2 + \frac{\rho}{2} |\nabla w|^2 \right] \, {\rm d}\sigma}_{:= I_{\varepsilon, \rho}^{3} } \\\nonumber
& +  \frac{N-2s}{2} \int_{\mathcal{O}_{\varepsilon, R, \rho}} y^{1-2s} |\nabla w|^2  \, dx \, dy. 
 \end{align}
 Now consider 
 $$
 I_{\varepsilon, R, \rho}^{1} = \int_{\partial \mathcal{O}^{1}_{\varepsilon, R, \rho}} y^{1- 2s} (- \partial_{y} w) (x, \nabla_{x} w) \, {\rm d}\sigma -\varepsilon^{2- 2s}\int_{\partial \mathcal{O}^{1}_{\varepsilon, R, \rho}}  |\partial_{y} w|^2 \, {\rm d}\sigma + 
  \frac{\varepsilon^{2- 2s}}{2} \int_{\partial \mathcal{O}^{1}_{\varepsilon, R, \rho}} |\nabla w|^2  \, {\rm d} \sigma.
 $$
We use   \cite[Lemma~4.1 and Proposition~3.7]{FF} in our framework, i.e.  we  set $m =0$ and $g(u):= {\theta u}/{|x|^{2s}} + u^{2^* -1} -u^{q} $ in \cite{FF}. Clearly by Proposition~\ref{l:8}, $g\in C^{0,\gamma}(B_{R} \setminus B_{\rho})$ for  $\gamma \in [0, 2-2s).$ Therefore using \eqref{11-12-20-1} 
 we can pass to the limit as $\varepsilon \rightarrow 0,$ in the 1st integral of $I_{\varepsilon, R, {\rho}}^{1}$ and obtain 
 
 \begin{equation}\label{NE1}
\int_{\partial \mathcal{O}^{1}_{\varepsilon, R, \rho}} y^{1- 2s} (- \partial_{y} w) (x, \nabla_{x} w) \, {\rm d}\sigma 
\rightarrow k_s \int_{B^{\prime}_{R, \rho} := \{ \rho^2 < |x|^2 < R^2\}}  f(u) (x, \nabla u) \, dx, \quad \mbox{as} \ \varepsilon \rightarrow 0.
 \end{equation}
 Moreover given any $r>0$  we show that there exists a sequence $\varepsilon_n\rightarrow 0$ such that 
 \begin{equation}\label{NE2}
\lim_{n\rightarrow +\infty}\left (\varepsilon_n^{2- 2s}\int_{|x|^2<r^2}  |\partial_{y} w(x,\varepsilon_n)|^2 \, {\rm d}x + 
  \frac{\varepsilon_n^{2- 2s}}{2} \int_{|x|^2<r^2} |\nabla w(x, \varepsilon_n)|^2  \, {\rm d}x\right )=0.
 \end{equation}
 Indeed by contradiction we would get 
 \begin{equation}\nonumber
\liminf_{\varepsilon\rightarrow 0}\left (\varepsilon^{2- 2s}\int_{|x|^2<r^2}  |\partial_{y} w(x,\varepsilon)|^2 \, {\rm d}x + 
  \frac{\varepsilon^{2- 2s}}{2} \int_{|x|^2<r^2} |\nabla w(x, \varepsilon)|^2  \, {\rm d}x\right )\geq C>0,
 \end{equation}
 namely there exists $\bar \varepsilon>0$ such that 
 \begin{equation}\nonumber
\varepsilon^{1- 2s}\int_{|x|^2<r^2}  |\partial_{y} w(x,\varepsilon)|^2 \, {\rm d}x + 
 \frac{\varepsilon^{1- 2s}}{2} \int_{|x|^2<r^2} |\nabla w(x, \varepsilon)|^2  \, {\rm d}x\geq \frac{C}{2\varepsilon},
 \end{equation}
for any $\varepsilon\in (0,\bar \varepsilon)$.  This is a contradiction because integrating on $(0,\bar\varepsilon)$, $w$ would not belong to $\dot H^1(\mathbb R^{N+1}_+; y^{1-2s})$. Therefore combining \eqref{NE1} and \eqref{NE2}, replacing $\varepsilon $ with $\varepsilon_n$  we conclude 
 \begin{equation}\label{11-12-20-2}
  I_{\varepsilon_n, R, \rho}^{1} \rightarrow k_s \int_{B^{\prime}_{R, \rho} := \{ \rho^2 < |x|^2 < R^2\}}  f(u) (x, \nabla u) \, dx, \quad \mbox{as} \   n\rightarrow +\infty
 \end{equation}
 and 
 \begin{equation}\nonumber
 I_{\varepsilon_n, R}^{2} \rightarrow \int_{S^+_R:=\{\mathbb R^{N+1}_+ \cap |z|^2=R^2\}} y^{1-2s} \left[ R |\partial_{\vec\eta}w |^2 - \frac{R}{2} |\nabla w|^2 \right] \, {\rm d}\sigma, \quad \mbox{as} \   n\rightarrow +\infty
 \end{equation}
 \begin{equation}\nonumber
 I_{\varepsilon_n, \rho}^{3} \rightarrow \int_{S^+_\rho:=\{\mathbb R^{N+1}_+ \cap |z|^2=\rho^2\}} y^{1-2s} \left[ -\rho |\partial_{\vec\eta}w |^2 + \frac{\rho}{2} |\nabla w|^2 \right] \, {\rm d}\sigma, \quad \mbox{as} \   n\rightarrow +\infty.
 \end{equation}
Moreover we have
 \begin{align*}
I_{R}^{2} := \left |\int_{S^+_R} y^{1-2s} \left[ R |\partial_{\vec\eta}w |^2 - \frac{R}{2} |\nabla w|^2 \right] \, {\rm d}\sigma \right |\leq 2 R \int_{S^+_R} y^{1-2s} |\nabla w|^2 \, {\rm d} \sigma.
 \end{align*}
We claim that there exists a sequence  $R_{n'} \rightarrow +\infty,$  such that 
 $$
 \lim_{{n'}\rightarrow +\infty}2R_{n'} \int_{S^+_{R_{n'}}} y^{1- 2s} |\nabla w|^2 \, {\rm d}\sigma =0.
 $$
As we did above, if  the claim is not true, then there exists $\bar{R} > 0$ and $C > 0$ such that 
 $$
 \int_{S^+_{R}} y^{1-2s} |\nabla w|^2 \, {\rm d}\sigma \geq \frac{C}{R},$$
 for any $R \in (\bar{R},+\infty)$.
 This immediately implies that 
 \begin{align*}
 \int_{\mathbb{R}^{N+1}_{+}} y^{1-2s} |\nabla w|^2 \, dx \, dy  & 
 \geq \int_{\bar{R}}^\infty \int_{S^+_{R}} y^{1-2s} |\nabla w|^2 {\rm d}\sigma \, dR \\
 & \geq \int_{\bar{R}}^{\infty} \frac{C}{R} \, dR = + \infty,
 \end{align*}
which is a contradiction to the fact that $w\in \dot H^1(\mathbb R^{N+1}_+; y^{1-2s}) $ and hence the claim holds. Hence
 \begin{equation*}
\lim_{{n'} \rightarrow \infty}   I_{R_{n'}}^{2} = 0
\end{equation*}
and therefore
\begin{equation}\label{11-12-20-4}
\lim_{{n'} \rightarrow \infty} \lim_{n \rightarrow \infty}  I_{\varepsilon_n, R_{n'}}^{2} = 0
\end{equation}
In the same way, we define
 \begin{eqnarray}\nonumber
I_{\rho}^{3}:=\left|\int_{S^+_\rho:=\{\mathbb R^{N+1}_+ \cap |z|^2=\rho^2\}} y^{1-2s} \left[ -\rho |\partial_{\vec\eta}w |^2 + \frac{\rho}{2} |\nabla w|^2 \right] \, {\rm d}\sigma\right|\leq 
2\rho\int_{S^+_\rho} y^{1-2s} |\nabla w|^2 \, {\rm d}\sigma
 \end{eqnarray}
Since  $w \in\dot H^1(\mathbb R^{N+1}_+; y^{1-2s})$,  arguing similarly as in the case of $I^2_{R}$, we conclude there exists a sequence $\rho_{{n''}} \rightarrow 0,$ such that 
\begin{equation}\label{rho-1}
\lim_{{n''}\rightarrow +\infty}2\rho_{n''} \int_{\partial S^{+}_{\rho_{n''}}}  y^{1-2s} |\nabla w|^2 \, {\rm d}\sigma =0
\end{equation}
and then
 \begin{equation}\label{111-12-20-4}
\lim_{{n''} \rightarrow \infty}\lim_{{n} \rightarrow \infty}    I_{\varepsilon_n, \rho_{n''}}^{3} = 0.
\end{equation}
Using integration by parts in \eqref{11-12-20-2} and recall \eqref{eq:eva2} we obtain 
 \begin{align}\label{NE3}
 \int_{B^{\prime}_{R, \rho} := \{ \rho^2 < |x|^2 < R^2\}}  f(u) (x, \nabla u) \, dx & = \int_{B^{\prime}_{R, \rho}} \bigg(\mbox{div} \big(x F(u)\big) +s\theta\frac{u^2}{|x|^{2s}} - N F(u)\bigg) \, dx \notag \\
 & = R \int_{\partial B_{R}^{\prime}} F(u) \, {\rm d} \sigma - \rho \int_{\partial B_{\rho}^{\prime}} F(u) \, {\rm d} \sigma \\\nonumber
 & + s\theta\int_{B^{\prime}_{R, \rho}}\frac{u^2}{|x|^{2s}} \, dx - N \int_{B^{\prime}_{R, \rho}} F(u) \, dx.
 \end{align}
Furthermore, letting $\rho_{n''} \rightarrow 0$ in  \eqref{NE3} along the same sequence as above (or eventually taking a subsequence),  arguing as above  using  Hardy \eqref{Hardy} inequality and the fact that $F(u) \in L^{1}(\Rn)$,  we get
\begin{align}\label{eq:vasc1}
 \lim_{{n''}\rightarrow +\infty} \int_{B^{\prime}_{R,\rho_{n''}} := \{ \rho_{{n''}}^2 < |x|^2 < R^2\}}  f(u) (x, \nabla u) \, dx &=
  R \int_{\partial B_{R}^{\prime}} F(u) \, {\rm d} S + s\theta\int_{B^{\prime}_{R}}\frac{u^2}{|x|^{2s}} \, dx  - N \int_{B^{\prime}_{R}} F(u) \, dx.
\end{align}
Similarly,  along the same sequence (or eventually on a subsequence)  $\{ R_{n'} \}$ (chosen in $I^2_{R}$)   we obtain
 \begin{align}\label{eq:vasc2}
 \lim_{n'\rightarrow +\infty} \lim_{{n''}\rightarrow +\infty}\int_{B^{\prime}_{R_{n'},\rho_{n''}} := \{ \rho_{{n''}}^2 < |x|^2 < R_{{n'}}^2\}}  f(u) (x, \nabla u) \, dx &=
s\theta\int_{\mathbb R^N}\frac{u^2}{|x|^{2s}} \, dx  - N \int_{\mathbb R^N} F(u) \, dx.
\end{align}
 Combining equations \eqref{11-12-20-2},  \eqref{NE3}, \eqref{eq:vasc1} and \eqref{eq:vasc2} we obtain
\begin{equation}\label{11-12-20-5}
\lim_{R_{n'} \rightarrow +\infty}\lim_{{n''}\rightarrow + \infty} \lim_{n\rightarrow + \infty} I_{\varepsilon_n, R_{n'}, \rho_{n''}}^{1} =s k_s\theta\int_{\mathbb R^N}\frac{u^2}{|x|^{2s}} \, dx -N  k_s\int_{\Rn} F(u) \, dx.
\end{equation}
Further, using \eqref{11-12-20-3}, we also have
 \begin{align}\notag\label{NE4}
\lim_{R_{n'} \rightarrow +\infty}\lim_{{n''}\rightarrow + \infty} \lim_{n\rightarrow + \infty}\frac{N-2s}{2} \int_{\mathcal{O}_{\varepsilon_n, R_{n'}, \rho_{n''}}} y^{1-2s} |\nabla w|^2 \, dx \, dy
 & =  \frac{N-2s}{2} \int_{\mathbb{R}^{N+1}_{+}} y^{1-2s} |\nabla w|^2 \, dx \, dy \\ \notag
 & = \frac{N-2s}{2} k_s\|u\|_{\dot{H}^s(\Rn)}^2 \notag \\
 & = \frac{N-2s}{2} k_s \int_{\Rn} u f(u) \, dx.
 \end{align}
 Finally combining \eqref{eq:vasc0}, \eqref{11-12-20-4}, \eqref{111-12-20-4}, \eqref{11-12-20-5}, \eqref{NE4}, we get 
$$
\frac{N-2s}{2} \int_{\Rn} u f(u) \, dx =-s\theta\int_{\mathbb R^N}\frac{u^2}{|x|^{2s}} \, dx+ N \int_{\Rn} F(u) \, dx,
$$
namely 
$$ 
\frac{N-2s}{2} \int_{\Rn}\left(u^{2^*_s}-u^{q+1}\right)\, dx= N\int_{\Rn}\left(\frac{1}{2^*_s}u^{2^*_s}-\frac{1}{q+1}u^{q+1}\right)\, dx.
$$
 }
 \end{proof}
We are ready to give the 
{\begin{proof}[\bf {Proof of Theorem~\ref{thm-non-exist}}]
(i):  Let $u$ be a solution of \eqref{prop-nonexist}, then using \eqref{P-identity} we have 
$$
\frac{N - 2s}{2} \int_{\Rn} \left (u^{2^*_s} - u^{q + 1} \right ) \, dx
= N \int_{\Rn} \left (\frac{1}{2^{*}_{s}} u^{2^*_s} - \frac{1}{q + 1}u^{q + 1} \right ) \, dx.
$$
Rearranging the above terms, we get
$$
\left(\frac{N - 2s}{2} - \frac{N}{q + 1} \right) \int_{\Rn} u^{q + 1} \, dx = 0.
$$
Since $q > 2^*_s - 1$ this immediately implies that $u \equiv 0$. This proves the first part of the theorem. 

\

(ii): We  use a constrained minimization.
 Let us define the manifold $$\mathcal{N}:= \bigg\{ u\in \dot{H}^s(\R^N) \cap L^{q+1}(\R^N):  \int_{\R^N} |u|^{p+1}{dx}=1 \bigg\}.$$ Define the functional $F$ on $\dot{H}^{s}(\Rn)\cap L^{q+1}(\R^N)$ as
 $$
F(u):=  \frac{a_{N,s}}{4}\int_{\R^N}\int_{\Rn}\frac{|u(x)-u(y)|^2}{|x-y|^{N+2s}} \, dx\,  dy- \frac{\theta}{2}\int_{\Rn}\frac{|u|^2}{|x|^{2s}}\, dx +  \frac{1}{q+1}\int_{\R^N}|u|^{q+1}\, dx
$$
and  define
\begin{equation}\label{eq:peach}
\mathcal{K}:=\inf_{\mathcal{N}}F(u).
\end{equation}

Let $u_{n}$ be a minimizing sequence in $\mathcal{N}$ such that
$$F(u_{n}) \to \mathcal{K} \text{ with } \int_{\R^N} |u_{n}|^{p+1}\,dx=1.$$
We assume that $u_{n} \geq 0$ a.e. in $\mathbb R^N$. This is not restrictive, since we can consider $|u_{n}|$ as a minimizing sequence. 
Since $\theta<\Lambda_{N,s}$, by Remark \ref{r:1} we have that the quantity 
$$\mathcal Q(u):=\frac{a_{N,s}}{2}\int_{\R^N}\int_{\Rn}\frac{|u(x)-u(y)|^2}{|x-y|^{N+2s}} \,{dx \, dy}- \theta\int_{\Rn}\frac{|u|^2}{|x|^{2s}}\,dx$$ is an equivalent norm in $\dot{H}^s(\Rn)$. Hence it follows that   $\{u_n\}$ is a bounded sequence in $\dot{H}^s(\Rn)\cap L^{q+1}(\Rn)$. Therefore, there exists $u\in \dot{H}^s(\Rn)\cap L^{q+1}(\Rn)$ such that $u_n \rightharpoonup u$ in $\dot{H}^s(\Rn)$ and  $u_n \rightharpoonup u$ in $L^{q+1}(\Rn)$. Consequently  $u_{n}\to u$  a. e. in $\mathbb R^N$, by the Sobolev embedding.

By using a Polya-Szeg\"o type  inequality, see~\cite{YJP}, we have that
\begin{equation}\nonumber
\frac{a_{N,s}}{2}\int_{\mathbb{R}^N}\int_{\mathbb{R}^N}\frac{|u_n(x)-u_n(y)|^2}{|x-y|^{N+2s}}\,dx\,dy\geq\frac{a_{N,s}}{2}\int_{\mathbb{R}^N}\int_{\mathbb{R}^N}\frac{|u^*_n(x)-u^*_n(y)|^2}{|x-y|^{N+2s}}\,dx \, dy,
\end{equation}
where by $f^*$ we denote the decreasing rearrangement of a measurable function $f$. 
Therefore via a  rearrangement technique and without loss of generality, we can assume that $u_n$ is  radially symmetric and decreasing. By \cite[Lemma 6.1]{BM}, for every $R:=|x|>0$ we deduce
\begin{equation}\label{eq:alib1}
u_n(R)\leq \frac{C}{R^{\frac{N-2s}{2}}},
\end{equation}
where $C$ is a positive constant that does not depend on $n$. Moreover for every bounded  measurable set $S\subset \mathbb R^N$ we have 
\begin{equation}\label{eq:alib2}
\int_{S} |u_n|^{p+1}\, dx\leq |S|^{\frac{q-p}{q+1}}\left(\int_{S}|u_n|^{q+1}\, dx \right)^{\frac{p+1}{q+1}}\leq C  |S|^{\frac{q-p}{q+1}},
\end{equation}
where $C$ is a positive constant that does not depend on $n$, since $\{u_n\}$ is a bounded sequence in $\dot{H}^s(\Rn)\cap L^{q+1}(\Rn)$. 
Using \eqref{eq:alib2} together with Vitali's theorem we see that $u_{n} \to u \, \text{in}\,\, L^{p+1}(B_R(0))$ for any $R>0$ and 
using equations \eqref{eq:alib1} (recall that $p>2^*_s -1$) it follows that
$u_{n} \to u \,  \text{in}\,\, L^{p+1}(\Rn\setminus B_R(0))$. Therefore,
$$u_{n} \to u \quad  \text{in}\,\, L^{p+1}(\R^N)$$ 
and as a consequence
$$\displaystyle\int_{\R^N }|u|^{p+1}\, dx=1,$$
namely $u\in \mathcal{N}$. Further, using Fatou's lemma and that fact that $u_n \rightharpoonup u$ in $\dot{H}^s(\Rn)$, 
it can be shown that
\begin{eqnarray*}&&\mathcal{K}=   \liminf_{n \to \infty } \left(\frac{1}{2}\mathcal Q(u_n)+  \frac{1}{q+1}\int_{\R^N}|u_{n}|^{q+1}\,dx\right) \\
&&=\liminf_{n \to \infty } \Bigg(\frac{1}{2}\mathcal Q(u)+\frac{1}{2}\mathcal Q(u_n-u) \\
&&+
\frac{a_{N,s}}{2}\int_{\Rn}\int_{\Rn}\frac{(u(x)-u(y))((u_n-u)(x)-(u_n-u)(y))}{|x-y|^{N+2s}}\,dx\, dy-\theta\int_{\Rn}\frac{u(u_n-u)}{|x|^{2s}}dx\\
&&
+  \frac{1}{q+1}\int_{\R^N}|u_{n}|^{q+1}\,dx\Bigg) \\
&&\geq  \left(\frac{1}{2}\mathcal Q(u)+  \frac{1}{q+1}\int_{\R^N}|u|^{q+1}\, dx\right)\\
&&=  F(u) .\end{eqnarray*}
This proves $F(u)=\mathcal{K}$.  Applying the Lagrange multiplier rule, we obtain that $u$ satisfies
\be\no (- \De)^s u-\theta\frac{u}{|x|^{2s}}+ u^q=  \lambda u^p, \ee
for some $\lambda>0$. 
Finally,  setting 
\begin{equation}\label{eq:prscus24}
v(x)=\la^{-\f{1}{q-p}}u(\la^{-\f{q-1}{2s(q-p)}}x),
\end{equation} it holds 
\begin{equation}\label{eq:prscus22}(-\De)^s v-\theta\frac{v}{|x|^{2s}}=v^p-v^q \quad\text{in}\,\,\Rn.
\end{equation} 
The solution $v$ obviously is non negative  in $\mathbb R^N$. We show that actually $v>0$  in $\mathbb R^N$. In fact the limit function $u$ is radial and radial decreasing since by construction,  the sequence $\{u_n\}$ is radial and radial decreasing. From \eqref{eq:prscus24} the solution $v$ does inherit the same properties. Therefore if by contradiction there exists some point $x_0\in \mathbb R^N$ such that $v(x_0)=0$, the support of $v$ would be contained in a ball, say of radius $R_0$. Let us take a non negative test function $\varphi_{c}\in C^{\infty}(\mathbb R^N)$ with compact support $\mathcal K$ contained in $\mathbb R^N\setminus B_{2R_0}$. Using equation \eqref{eq:prscus22} we deduce
\begin{equation}\nonumber
\frac{a_{N,s}}{2}\int_{\Rn}\int_{\Rn}\frac{(v(x)-v(y))(\va_c(x)-\va_c(y))}{|x-y|^{N+2s}}\,dx\, dy=0,
\end{equation}
that is 
\begin{equation}\label{eq:prseagle}
\int_{\mathcal K}\int_{\mathcal K}\frac{(v(x)-v(y))(\va_c(x)-\va_c(y))}{|x-y|^{N+2s}}\,dx\, dy + 2\int_{\mathcal K}\int_{\mathbb R^N\setminus\mathcal K}\frac{(v(x)-v(y))(\va_c(x)-\va_c(y))}{|x-y|^{N+2s}}\,dx\, dy=0.
\end{equation}
Because $v\equiv0$ on $\mathcal K$, the first term of \eqref{eq:prseagle}
\begin{equation}\nonumber
\int_{\mathcal K}\int_{\mathcal K}\frac{(v(x)-v(y))(\va_c(x)-\va_c(y))}{|x-y|^{N+2s}}\,dx\, dy=0.
\end{equation}
The second term 
\begin{equation}\nonumber
2\int_{\mathcal K}\int_{\mathbb R^N\setminus\mathcal K}\frac{(v(x)-v(y))(\va_c(x)-\va_c(y))}{|x-y|^{N+2s}}\,dx\, dy\,=2\int_{\mathcal K}\int_{\mathbb R^N\setminus\mathcal K}\frac{(-v(y))(\va_c(x))}{|x-y|^{N+2s}}\,dx\, dy  <0,\end{equation}
since $v,\varphi_c\geq 0$. This is a contradiction with \eqref{eq:prseagle}.
\end{proof}
\section{Qualitative properties  for solutions to \eqref{eq:a3}}
In this section we first prove Theorem \ref{thm:rad}. We show that all the solutions of \eqref{eq:a3} (see also Remark~\ref{rem:giurie}) are radial  and radially decreasing with respect to the origin, as stated in Theorem~\ref{thm:rad}. The proof  will be carried out  exploiting the moving plane method. Without loss of  generality, we start considering the $x_1$-direction. 
We denote a point $x\in\mathbb{R}^N$ as $x=(x_1,x_2,\cdots,x_N)$ and, for  any~$\lambda \in \mathbb{R}$, we set
\begin{equation}\label{sigma l}
\Sigma_{\lambda}:=\Big\{x\in \mathbb{R}^N : x_1< \lambda\Big\}
\end{equation}
and
\begin{equation}\label{T l}
T_{\lambda}:=\Big\{x\in \mathbb{R}^N : x_1=\lambda\Big\}.
\end{equation}
For any~$\lambda \in \mathbb{R}$, we also set
\begin{equation}\label{x lambda}
x_\lambda:=(2\lambda-x_1,x_2,\cdots,x_N)
\end{equation}
and we define
\begin{equation}\label{eq:defulamb}
u_{\lambda}(x):=u(x_\lambda).
\end{equation}
With these definitions, if $u$ is a positive solution to \eqref{eq:a3}, then $u_\lambda$ weakly satisfies 
\begin{equation}
  \tag{$\mathcal P_\lambda$}\label{eq:a3lambda}\left\{\begin{aligned}
 (-\Delta)^s u_\lambda -\theta\frac{u_\lambda}{|x_\lambda|^{2s}} &=u_\lambda^p - u_\lambda^q \quad\text{in }\quad \Rn,
 \\
      u_\lambda &>0 \quad\text{in }\quad \Rn,\\
            u_\lambda &\in \dot{H}^s(\Rn)\cap L^{q+1}(\Rn).
 \end{aligned}
  \right.
 \end{equation}
Indeed for any $\varphi\in\dot{H}^s(\R^N) \cap L^{q+1}(\Rn)$, we have
 \begin{eqnarray*} 
&&\frac{a_{N,s}}{2}\int_{\mathbb{R}^N}\int_{\mathbb{R}^N}\frac{(u_\lambda(x)-u_\lambda(y))(\varphi(x)-\varphi(y))}{|x-y|^{N+2s}}\,dx\,dy\\
&=&\frac{a_{N,s}}{2}\int_{\mathbb{R}^N}\int_{\mathbb{R}^N}\frac{(u(x_\lambda)-u(y_\lambda))(\varphi(x)-\varphi(y))}{|x-y|^{N+2s}}\,dx\,dy\\
&=&\frac{a_{N,s}}{2}\int_{\mathbb{R}^N}\int_{\mathbb{R}^N}\frac{(u(t)-u(z))(\varphi(t_\lambda)-\varphi(z_\lambda))}{|t_\lambda-z_\lambda|^{N+2s}}\,dt\, dz\\
&=&\frac{a_{N,s}}{2}\int_{\mathbb{R}^N}\int_{\mathbb{R}^N}\frac{(u(t)-u(z))(\varphi(t_\lambda)-\varphi(z_\lambda))}{|t-z|^{N+2s}}\,dt\,dz\\
&=&\frac{a_{N,s}}{2}\int_{\mathbb{R}^N}\int_{\mathbb{R}^N}\frac{(u(t)-u(z))(\varphi_\lambda(t)-\varphi_\lambda(z))}{|t-z|^{N+2s}}\,dt\,dz\\
&=& \theta\int_{\mathbb{R}^N}\frac{u(t)}{|t|^{2s}}\,\varphi_\lambda(t)\,dt +\int_{\mathbb{R}^N}(u^{p}(t)-u^{q}(t))\varphi_\lambda(t)\,dt\\
&=&\theta\int_{\mathbb{R}^N}\frac{u_\lambda(x)}{|x_\lambda|^{2s}}\,\varphi(x)\,dx +\int_{\mathbb{R}^N}(u_\lambda^{p}(x)-u_\lambda^{q}(x))\varphi(x)\,dx,
\end{eqnarray*}
where the changes of variables~$t=x_\lambda$ and~$z=y_\lambda$
and~\eqref{1.1bis} were used. Notice that,   if~$\varphi\in\dot{H}^s(\R^N)\cap L^{q+1}(\Rn)$,
then~$\varphi_\lambda\in\dot{H}^s(\R^N)\cap L^{q+1}(\Rn)$ and so~$\varphi_\lambda$ can be used as a test function in~\eqref{1.1bis}. 
\begin{proof}[\bf {Proof of Theorem~\ref{thm:rad}}]
Let $\lambda<0$ and we define the  following function
\begin{equation}\label{defw}
w_\lambda(x):=\left\{\begin{array}{ll}
(u-u_\lambda)^+(x), & \quad {\mbox{ if }} x\in\Sigma_\lambda, \\
-(u-u_\lambda)^-(x), & \quad {\mbox{ if }}x\in\R^N\setminus\Sigma_\lambda,
\end{array} \right.
\end{equation}
where~$(u-u_\lambda)^+:=\max\{u-u_\lambda,0\}$ and~$(u-u_\lambda)^-:=-\min\{u-u_\lambda,0\}$. Note that $w_\lambda$ is anti-symmetric w.r.t. $T_\lambda$. We set
\begin{equation}\label{4.3bis}
\begin{split}
&\mathcal{S}_\lambda\,:=\, \big({\rm supp} \,\,w_\lambda\big)\cap \Sigma_\lambda, \qquad\qquad\qquad\qquad
\mathcal{S}_\lambda^c\,:=\,\Sigma_\lambda\setminus \mathcal{S}_\lambda,\\
&\mathcal{D}_\lambda\,:=\, \big({\rm supp} \,\,w_\lambda\big)\cap \Big(\R^N\setminus\Sigma_\lambda\Big),\qquad\qquad\,
\mathcal{D}_\lambda^c\,:=\,\Big(\R^N\setminus\Sigma_\lambda\Big)\setminus \mathcal{D}_\lambda\,.\\
\end{split}
\end{equation}
Using \eqref{eq:defulamb},  it is not difficult to see
\begin{equation}\label{refle}
{\mbox{$\mathcal D_\lambda$ is the reflection of~$\mathcal S_\lambda$.}}
\end{equation}
{\bf Step 1:} We claim  that
\begin{equation}\label{claim}
{\mbox{$w_\lambda\equiv 0$ for $\lambda<0$ with~$|\lambda|$ sufficiently large.}}
\end{equation}
We start noticing that,  the function~$w_\lambda$ defined
in~\eqref{defw} belongs to~$\dot{H}^s(\mathbb{R}^N)\cap L^{q+1}(\Rn)$. Therefore, we can use it as test
function in the weak  formulations of~\eqref{eq:a3} and~\eqref{eq:a3lambda}.
We have
\begin{eqnarray}\label{eq:M1}\\\nonumber
&&\frac{a_{N,s}}{2}\int_{\mathbb{R}^N}\int_{\mathbb{R}^N}\frac{(u(x)-u(y))(w_\lambda(x)-
w_\lambda(y))}{|x-y|^{N+2s}}\,dx\,dy-\theta\int_{\mathbb{R}^N}\frac{u}{|x|^{2s}}w_\lambda \,dx\\\nonumber
&&=\int_{\mathbb{R}^N}(u^{p}-u^{q})w_\lambda \,dx,
\\\nonumber
&&\frac{a_{N,s}}{2}\int_{\mathbb{R}^N}\int_{\mathbb{R}^N}\frac{(u_\lambda(x)-u_\lambda(y))
(w_\lambda(x)-w_\lambda(y))}{|x-y|^{N+2s}}\,dx\,dy-\theta\int_{\mathbb{R}^N}\frac{u_\lambda}
{|x_\lambda|^{2s}}w_\lambda \,dx\\\nonumber&&=\int_{\mathbb{R}^N}(u_\lambda^{p}-u_\lambda^{q})w_\lambda \,dx.
\end{eqnarray}
Subtracting the two equations in \eqref{eq:M1} we obtain
\begin{eqnarray}\label{eq:M2}\\\nonumber
&&\frac{a_{N,s}}{2} \int_{\mathbb{R}^N}\int_{\mathbb{R}^N}\frac{((u(x)-u_\lambda(x))-(u(y)-u_\lambda(y)))
(w_\lambda(x)-w_\lambda(y))}{|x-y|^{N+2s}}\,dx\,dy\\\nonumber
&=&\theta\int_{\mathbb{R}^N}\left(\frac{u}{|x|^{2s}}-\frac{u_\lambda}{|x_\lambda|^{2s}}\right)w_\lambda\, dx+\int_{\mathbb{R}^N}(u^{p}-u_\lambda^{p})w_\lambda\, dx- \int_{\mathbb{R}^N}(u^{q}-u_\lambda^{q})w_\lambda\, dx\\\nonumber
&\leq&\theta\int_{\mathbb{R}^N}\frac{(u-u_\lambda)}{|x|^{2s}}w_\lambda \, dx+\int_{\mathbb{R}^N}(u^{p}-u_\lambda^{p})w_\lambda\, dx- \int_{\mathbb{R}^N}(u^{q}-u_\lambda^{q})w_\lambda\, dx,
\end{eqnarray}
since $|x|\geq|x_\lambda|$ and~$w_\lambda\ge0$ in~$\Sigma_\lambda$ and~$|x|\leq|x_\lambda|$ and~$w_\lambda\le0$ outside $\Sigma_\lambda$.
On the other hand, we have
\begin{equation}\label{eq:M3}
\begin{split}
&\int_{\mathbb{R}^N}\int_{\mathbb{R}^N}
\frac{((u(x)-u_\lambda(x))-(u(y)-u_\lambda(y)))(w_\lambda(x)-w_\lambda(y))}{|x-y|^{N+2s}}\,dx\,dy\\
&=\int_{\mathbb{R}^N}\int_{\mathbb{R}^N}
\frac{(w_\lambda(x)-w_\lambda(y))^2}{|x-y|^{N+2s}}\,dx\,dy\\
& +\int_{\mathbb{R}^N}\int_{\mathbb{R}^N}
\frac{((u(x)-u_\lambda(x))-(u(y)-u_\lambda(y))-(w_\lambda(x)-w_\lambda(y)))(w_\lambda(x)-w_\lambda(y))}{|x-y|^{N+2s}}\,dx\,dy\\
&= \int_{\mathbb{R}^N}\int_{\mathbb{R}^N}
\frac{(w_\lambda(x)-w_\lambda(y))^2}{|x-y|^{N+2s}}\,dx\,dy+\int_{\mathbb{R}^N}\int_{\mathbb{R}^N}
\frac{\mathcal{G}(x,y)}{|x-y|^{N+2s}}\,dx\,dy,
\end{split}
\end{equation}
where
\begin{equation}\nonumber
\mathcal{G}(x,y)\,:=\,((u(x)-u_\lambda(x))-(u(y)-u_\lambda(y))-
(w_\lambda(x)-w_\lambda(y)))(w_\lambda(x)-w_\lambda(y))\,.
\end{equation}
Arguing as in \cite{BMS, DMPS}, we deduce that 
\begin{equation}\nonumber
\int_{\mathbb{R}^N}\int_{\mathbb{R}^N}
\frac{\mathcal{G}(x,y)}{|x-y|^{N+2s}}\,dx\,dy\geq 0\,
\end{equation}
and therefore from \eqref{eq:M3}, it follows 
\begin{equation}\label{eq:M33}
\begin{split}
&\int_{\mathbb{R}^N}\int_{\mathbb{R}^N}
\frac{((u(x)-u_\lambda(x))-(u(y)-u_\lambda(y)))(w_\lambda(x)-w_\lambda(y))}{|x-y|^{N+2s}}\,dx\,dy
\\&\geq  \int_{\mathbb{R}^N}\int_{\mathbb{R}^N}
\frac{(w_\lambda(x)-w_\lambda(y))^2}{|x-y|^{N+2s}}\,dx\,dy.
\end{split}
\end{equation}
To estimate  the RHS of \eqref{eq:M2}, we first note that
$$(u-u_\lambda)w_\lambda=w_\lambda^2 \quad\text{in}\quad\Rn.$$
For the first term on the RHS of \eqref{eq:M2},  we use Hardy inequality and we get
\begin{equation}\label{eq:M4}
\theta\int_{\mathbb{R}^N}\frac{(u-u_\lambda)}{|x|^{2s}}w_\lambda \, dx=\theta\int_{\mathbb{R}^N}\frac{w_\lambda^2}{|x|^{2s}} \,dx\leq \frac{\theta}{\Lambda_{N,s}}\frac{a_{N,s}}{2}\int_{\mathbb{R}^N}\int_{\mathbb{R}^N}\frac{(w_\lambda(x)-w_\lambda(y))^2}{|x-y|^{N+2s}}\,dx\,dy.
\end{equation}
To estimate the second term on the RHS of \eqref{eq:M2}, we observe   that for any $t>1$
$$0\leq\frac{u^t-u_\lambda^t}{u-u_\la}\leq t\max\{u^{t-1}, u_\la^{t-1}\}.$$ Therefore,  recalling \eqref{4.3bis} and \eqref{refle}, we have 
\begin{eqnarray}\label{eq:1mat}
&&\int_{\mathbb{R}^N}(u^{p}-u_\lambda^{p})w_\lambda\, dx- \int_{\mathbb{R}^N}(u^{q}-u_\lambda^{q})w_\lambda\, dx\\\nonumber
&&=\int_{\mathbb{R}^N}\frac{u^{p}-u_\lambda^{p}}{u-u_\lambda}w^2_\lambda\, dx- \int_{\mathbb{R}^N}\frac{u^{q}-u_\lambda^{q}}{u-u_\lambda}w^2_\lambda\, dx\\\nonumber
&&\leq\int_{\mathcal{S}_\lambda}\frac{u^{p}-u_\lambda^{p}}{u-u_\lambda}w^2_\lambda\, dx+\int_{\mathcal{D}_\lambda}\frac{u^{p}-u_\lambda^{p}}{u-u_\lambda}w^2_\lambda\, dx\\\nonumber
&&\leq p\int_{\mathcal{S}_\lambda}u^{p-1}\cdot w_\lambda^2\,dx
+p\int_{\mathcal{D}_\lambda}u_\lambda^{p-1}\cdot w_\lambda^2\,dx\\\nonumber
&&= 2p \int_{\mathcal{S}_\lambda}u^{p-1}\cdot w_\lambda^2\,dx,
\end{eqnarray}
where in the last line we exploited a change of variable.
By H\"older and Sobolev inequalities we deduce 
 \begin{eqnarray}\nonumber
  &&\int_{\mathcal{S}_\lambda}u^{p-1}\cdot w_\lambda^2\,dx
  \\\nonumber &&\leq \left(\int_{\mathcal{S}_\lambda}u^{(p-1)\frac{2_s^*}{2^*_s-2}}\,dx\right)^{\frac{2^*_s-2}{2_s^*}}
\left(\int_{\mathcal{S}_\lambda}w_\lambda^{2^*_s}\,dx\right)^{\frac{2}{2^*_s}} \leq \left(\int_{\mathcal{S}_\lambda}u^{(p-1)\frac{2_s^*}{2^*_s-2}}\,dx\right)^{\frac{2^*_s-2}{2_s^*}}
\left(\int_{\mathbb{R}^N} w_\lambda^{2^*_s}\,dx\right)^{\frac{2}{2^*_s}}\\\nonumber
&&\leq S \left(\int_{\mathcal{S}_\lambda}u^{(p-1)\frac{2_s^*}{2^*_s-2}}\,dx\right)^{\frac{2^*_s-2}{2_s^*}}
\int_{\mathbb{R}^N}\int_{\mathbb{R}^N} \frac{(w_\lambda(x)-w_\lambda(y))^2}{|x-y|^{N+2s}}\,dx\,dy.
 \end{eqnarray}
 Collecting the last inequality, \eqref{eq:M2}, \eqref{eq:M33}, \eqref{eq:M4} and \eqref{eq:1mat}, we finally obtain
 \begin{eqnarray}\label{eq:M6}\\\nonumber
&&\left(\frac{a_{N,s}}{2}-\frac{\theta}{\Lambda_{N,s}}\frac{a_{N,s}}{2}\right)\int_{\mathbb{R}^N}\int_{\mathbb{R}^N}\frac{(w_\lambda(x)-w_\lambda(y))^2}{|x-y|^{N+2s}}\,dx\,dy
\\\nonumber&\leq& 2pS\left(\int_{\mathcal{S}_\lambda}u^{(p-1)\frac{2_s^*}{2^*_s-2}}\,dx\right)^{\frac{2^*_s-2}{2_s^*}}\int_{\mathbb{R}^N}\int_{\mathbb{R}^N}\frac{(w_\lambda(x)-w_\lambda(y))^2}{|x-y|^{N+2s}}\,dx\,dy.
\end{eqnarray}
Since  $p> 2^*_s-1$ and \eqref{eq:qmaggioredip}, we get  
\begin{equation}\label{eq:marebrutto}\left(\int_{\mathbb R^N}u^{(p-1)\frac{2_s^*}{2^*_s-2}}\,dx\right)^{\frac{2^*_s-2}{2_s^*}}=\left(\int_{\mathbb R^N}u^{(p-1)\frac{N}{2s}}\,dx\right)^{\frac{2^*_s-2}{2_s^*}} < +\infty.
\end{equation}
Therefore there exists~$R>0$  such that
for~$\lambda<-R$ we have
$$ 2pS\left(\int_{\mathcal{S}_\lambda}u^{(p-1)\frac{2_s^*}{2^*_s-2}}\,dx\right)^{\frac{2^*_s-2}{2_s^*}}\leq
2pS\left(\int_{\Sigma_\lambda}u^{(p-1)\frac{2_s^*}{2^*_s-2}}\,dx\right)^{\frac{2^*_s-2}{2_s^*}}\le  \frac12\left(\frac{a_{N,s}}{2}-\frac{\theta}{\Lambda_{N,s}}\frac{a_{N,s}}{2}\right).$$
This and~\eqref{eq:M6} give that
$$\int_{\mathbb{R}^N}\int_{\mathbb{R}^N}\frac{(w_\lambda(x)-w_\lambda(y))^2}{|x-y|^{N+2s}}\,dx\,dy= 0\,.$$
This implies that  $w_\lambda$ is constant and the claim~\eqref{claim}
follows since~$w_\lambda$ is zero on $\{x_1=\lambda\}$.

Now we define the set
$$
\Lambda:=\{\lambda\in \mathbb{R}\,:\, u\leq u_\mu \,\text{in }\, \Sigma_\mu\,\, \forall \mu \leq \lambda \}.$$
Notice that~\eqref{claim} implies that~$\Lambda\neq\emptyset$,
and therefore we can consider
\begin{equation}\label{eq:sup}
\bar \lambda:=\sup\Lambda.
\end{equation}
{\bf Step 2:} We will show that
\begin{equation}\label{lambda zero}
\bar\lambda=0.
\end{equation}
Let us  assume by contradiction that $\bar\lambda <0$.
Now, in this case,  we are going to show that we can move the plane a little further to the right reaching a contradiction with the definition \eqref{eq:sup}.
Indeed, by continuity, we have that $u\leq u_{\bar\lambda}$ in~$\Sigma_{\bar\lambda}$ (say outside the reflected point of the origin $0_{\bar\lambda}$, since $u_{\bar\lambda}$ may be not well defined there). We want to prove that 
\begin{equation}\label{eq:mat22}
u<u_{\bar\lambda}\quad \text{in}\,\,\Sigma_{\bar\lambda}.
\end{equation}
The case $u\equiv u_{\bar\lambda}$ in~$\Sigma_{\bar\lambda}$, is  not possible because of the presence of the Hardy potential. Indeed if this would be the case, $|(-\Delta)^s u(0_{\bar\lambda})|<+\infty$, since in the point $0_{\bar\lambda}$ the solution $u$ is regular and consequently $|(-\Delta)^s u_{\bar\lambda}(0_{\bar\lambda})|<+\infty$ as well. Moreover $(-\Delta)^s u(0_{\bar\lambda})=(-\Delta)^s u_{\bar\lambda}(0_{\bar\lambda})$ and therefore  
\begin{equation}\label{eq:hardy}
\theta\frac{u(0_{\bar\lambda})}{|0_{\bar\lambda}|^{2s}} + u^p(0_{\bar\lambda}) - u^q(0_{\bar\lambda})
=\theta\frac{u_{\bar\lambda}(0_{\bar\lambda})}{|0_{\bar\lambda}|^{2s}} + u_{\bar\lambda}^p(0_{\bar\lambda}) - u_{\bar\lambda}^q(0_{\bar\lambda}).
\end{equation}
The right hand side of \eqref{eq:hardy} is well defined, i.e. 
$$ 
\left|\theta\frac{u_{\bar\lambda}(0_{\bar\lambda})}{|0_{\bar\lambda}|^{2s}} + u_{\bar\lambda}^p(0_{\bar\lambda}) - u_{\bar\lambda}^q(0_{\bar\lambda})\right|<+\infty
$$
if and only if $u_{\bar\lambda}(x)=O(|x|^{2s})$ if $x\rightarrow 0_{\bar\lambda}$. We point out that we are in the case $u\equiv u_{\bar\lambda}$ in~$\Sigma_{\bar\lambda}$ and therefore both $u$ and $u_{\bar\lambda}$ are $C^\infty$ in a neighborhood of $0_{\bar\lambda}$.  But this is a contradiction since the solution $u$ is positive in $\mathbb R^N$. 

Therefore, to prove \eqref{eq:mat22}, we assume by contradiction that there exists a point $\bar x$ in $ \Sigma_{\bar\lambda}\setminus \{0_{\bar\lambda}\}$ where
\begin{equation}\label{eq:contrddd}
u(\bar x)=u_{\bar \lambda}(\bar x).
\end{equation}
We fix now  $r>0$ such that $0_{\bar\lambda} \notin \overline{B_r}(\bar x )$ and $0 \notin \overline{B_r}(\bar x_{\bar\lambda})$. Then using  Proposition \ref{l:8} we have that for some  $0<\beta<2s$ it holds
\begin{eqnarray}\nonumber
\|u\|_{\mathcal{C}^{2s+\beta}\left(\overline{B_{\frac{r}{16}}}(\bar x)\right)}\leq C &\text {and}& \|u_{\bar\lambda}\|_{\mathcal{C}^{2s+\beta}\left(\overline{B_{\frac{r}{16}}}(\bar x)\right)}\leq C \quad \text{if } 0<s<1/2,\\\nonumber
\|u\|_{\mathcal{C}^{1,2s+\beta-1}\left(\overline{B_{\frac{r}{16}}}(\bar x)\right)}\leq C &\text {and}& \|u_{\bar\lambda}\|_{\mathcal{C}^{1,2s+\beta-1}\left(\overline{B_{\frac{r}{16}}}(\bar x)\right)}\leq C \quad \text{if } 1/2\leq s<1,
\end{eqnarray}
for some positive constant $C$. As consequence, using also \eqref{int-con},  we can write the pointwise formulation of the problems \eqref{eq:a3} and  \eqref{eq:a3lambda}  for  $u$ and $u_{\bar\lambda}$ respectively at the point $x=\bar x$. Therefore
\begin{eqnarray}\label{eq:1w_l}
&&(-\Delta)^su_{\bar\lambda}(\bar x)-(-\Delta)^su(\bar x)-\theta\left(\frac{u_{\bar\lambda}(\bar x)}{|\bar x_{\bar\lambda}|^{2s}}-\frac{u(\bar x)}{|\bar x|^{2s}}\right)\\\nonumber
&&= (u_{\bar\lambda}^p(\bar x) - u^p(\bar x)) - (u_{\bar\lambda}^q(\bar x)-u^q(\bar x)).
\end{eqnarray}
It is worth noticing that, by \eqref{eq:contrddd} and  the fact that $|\bar x_\lambda|<|\bar x|$ for $\lambda <0$, from \eqref{eq:1w_l} it follows that
\begin{equation}\label{eq:1w_l'}
(-\Delta)^su_{\bar\lambda}(\bar x)-(-\Delta)^su(\bar x)> 0.
\end{equation}

\noindent On the other hand, by \eqref{De-u}, \eqref{sigma l}-\eqref{eq:defulamb}, \eqref{eq:contrddd} and the fact that the function $u_{\bar\lambda}(y)-u(y)$ is odd with respect to the hyperplane $\partial \Sigma_{\bar\lambda}=T_{\bar\lambda}$, it follows that
\begin{eqnarray}\label{Auguri}
&&(-\Delta)^su_{\bar\lambda}(\bar x)-(-\Delta)^su(\bar x)\\\nonumber
&&=a_{N,s}\,{\rm P.V.}\int_{\mathbb{R}^N}\frac{u_{\bar\lambda}(\bar x)-u_{\bar\lambda}(y)}{|\bar x-y|^{N+2s}}dy-a_{N,s}\,{\rm P.V.}\int_{\mathbb{R}^N}\frac{u(\bar x)-u(y)}{|\bar x-y|^{N+2s}}dy\\\nonumber
&&=-a_{N,s}\,{\rm P.V.}\int_{\mathbb{R}^N}\frac{u_{\bar\lambda}(y)-u(y)}{|\bar x-y|^{N+2s}}dy\\\nonumber&&=-a_{N,s}\,{\rm P.V.}\int_{\Sigma_{\bar\lambda}}\frac{u_{\bar\lambda}(y)-u(y)}{|\bar x-y|^{N+2s}}dy-a_{N,s}\,{\rm P.V.}\int_{\mathbb{R}^N\setminus\Sigma_{\bar\lambda}}\frac{u_{\bar\lambda}(y)-u(y)}{|\bar x-y|^{N+2s}}dy
\\\nonumber &&=-a_{N,s}\,{\rm P.V.}\int_{\Sigma_{\bar\lambda}}\frac{u_{\bar\lambda}(y)-u(y)}{|\bar x-y|^{N+2s}}dy-a_{N,s}\,{\rm P.V.}\int_{\Sigma_{\bar\lambda}}\frac{u_{\bar\lambda}(y_{\bar\lambda})-u(y_{\bar\lambda})}{|\bar x-y_{\bar\lambda}|^{N+2s}}dy\\\nonumber
\\\nonumber &&=-a_{N,s}\,{\rm P.V.}\int_{\Sigma_{\bar\lambda}}\frac{u_{\bar\lambda}(y)-u(y)}{|\bar x-y|^{N+2s}}dy-a_{N,s}\,{\rm P.V.}\int_{\Sigma_{\bar\lambda}}\frac{u(y)-u_{\bar\lambda}(y)}{|\bar x-y_{\bar\lambda}|^{N+2s}}dy\\\nonumber
&&=-a_{N,s}\,{\rm P.V.}\int_{\Sigma_{\bar\lambda}}(u_{\bar\lambda}(y)-u(y))\left(\frac{1}{|\bar x-y|^{N+2s}}-\frac{1}{|\bar x-y_{\bar\lambda}|^{N+2s}}\right)dy.
\end{eqnarray}
Since $|\bar x-y|<|\bar x -y_{\bar\lambda}|$ for $\bar x,y\in \Sigma_{\bar\lambda}$ and $u\not \equiv u_{\bar\lambda}$, $u\leq u_{\bar\lambda}$ in  $\Sigma_{\bar\lambda}$, from \eqref{Auguri}, {by continuity}, we have
\begin{equation}\nonumber
(-\Delta)^su_{\bar\lambda}(\bar x)-(-\Delta)^su(\bar x)\,<\,0,\end{equation}
a contradiction with \eqref{eq:1w_l'}. Since  $\bar x$ is an arbitrary point in~$\Sigma_{\bar\lambda}\setminus\{0_{\bar\lambda}\}$, this implies~\eqref{eq:mat22}.

Now notice that, for $\lambda^*<0$ given,  the inequality in~\eqref{eq:M6} holds for any~$\lambda\leq \lambda^*<0$
for a  constant~$C=2p$ that  is independent of~$\lambda$.
Then, since~$\bar\lambda<0$, there exists~$\varepsilon_1>0$
such that~$\bar\lambda+2\varepsilon<0$ for any~$\varepsilon\in(0,\varepsilon_1)$. Let us set $\lambda^*=\bar\lambda+2\varepsilon_1$. 
Recalling the notation introduced in~\eqref{defw} and~\eqref{4.3bis},
we consider the function~$w_{\bar\lambda+\varepsilon}$.
Using the same notation as above let us consider $w_{\bar\lambda+\varepsilon}$ so that
\begin{equation}\nonumber
\text{supp}\,\,w_{\bar\lambda+\varepsilon}\,\equiv\,\mathcal{S}_{\bar\lambda+\varepsilon}
\cup\mathcal{D}_{\bar\lambda+\varepsilon}\,.
\end{equation}
Exploiting the fact that $u<u_{\bar\lambda}$ in $\Sigma_{\bar\lambda}$ and the fact that the solution $u$ is continuous in $\mathbb{R}^N\setminus\{0\}$ (resp. $u_\lambda$ is continuous in $\mathbb{R}^N\setminus\{0_{\bar\lambda}\}$), we deduce that:
\noindent given any $R>0$ (large) and $\delta,\delta_1>0$ (small) we can fix $\bar\varepsilon=\bar\varepsilon(R,\delta,\delta_1)>0$ and $\bar\varepsilon<\varepsilon_1$, such that (arguing by continuity) 
\begin{equation}
(\mathcal{S}_{\bar\lambda+\varepsilon}\cap\Sigma_{\bar\lambda-\delta})
\subset \big(\mathbb{R}^N\setminus (B_R(0)\setminus B_{\delta_1}(0_{\bar\lambda}) \big)\cup B_{\delta_1}(0_{\bar\lambda})\qquad\text{for any}\quad0\leq \varepsilon\leq\bar\varepsilon\,.
\end{equation}
We repeat now the argument above using  $w_{\bar\lambda+\varepsilon}$ as test function in the same fashion as we did using $w_\lambda$ and get again
\begin{eqnarray}\label{eq:M7}\\\nonumber
&&\left(\frac{a_{N,s}}{2}-\frac{\theta}{\Lambda_{N,s}}\frac{a_{N,s}}{2}\right)
\int_{\mathbb{R}^N}\int_{\mathbb{R}^N}\frac{(w_{\bar\lambda+\varepsilon}(x)-w_{\bar\lambda+\varepsilon}(y))^2}{|x-y|^{N+2s}}\,dx\,dy
\\\nonumber&\leq&2pS
\left(\int_{\mathcal{S}_{\bar\lambda+\varepsilon}}u^{(p-1)\frac{2_s^*}{2^*_s-2}}\,dx\right)^{\frac{2^*_s-2}{2_s^*}}
\int_{\mathbb{R}^N}\int_{\mathbb{R}^N}\frac{(w_{\bar\lambda+\varepsilon}(x)-w_{\bar\lambda+\varepsilon}(y))^2}{|x-y|^{N+2s}}\,dx\,dy.
\end{eqnarray}
Since
\begin{equation}\nonumber
\begin{split}
\int_{\mathcal{S}_{\bar\lambda+\varepsilon}}u^{(p-1)\frac{2_s^*}{2^*_s-2}}\,dx &\leq\int_{\mathcal{S}_{\bar\lambda+\varepsilon}\cap\Sigma_{\bar\lambda-\delta}}u^{(p-1)\frac{2_s^*}{2^*_s-2}}\,dx
+\int_{\Sigma_{\bar\lambda+\varepsilon}\setminus \Sigma_{\bar\lambda-\delta}}u^{(p-1)\frac{2_s^*}{2^*_s-2}}\,dx\\
&\leq\int_{\mathbb{R}^N\setminus B_R(0)}u^{(p-1)\frac{2_s^*}{2^*_s-2}}\,dx
+\int_{B_{\delta_1}(0_{\bar\lambda})}u^{(p-1)\frac{2_s^*}{2^*_s-2}}\,dx
+\int_{\Sigma_{\bar\lambda+\varepsilon}\setminus \Sigma_{\bar\lambda-\delta}}u^{(p-1)\frac{2_s^*}{2^*_s-2}}\,dx\\
\end{split}
\end{equation}
for $R$ large and $\delta,\delta_1$ small (recall \eqref{eq:marebrutto}), choosing $\bar\varepsilon(R,\delta,\delta_1)$ as above and eventually reducing it, we can assume that
$$2pS\left(\int_{\mathcal{S}_{\bar\lambda+\varepsilon}}u^{(p-1)\frac{2_s^*}{2^*_s-2}}\,dx\right)^{\frac{2^*_s-2}{2_s^*}}<
\left(\frac{a_{N,s}}{2}-\frac{\theta}{\Lambda_{N,s}}\frac{a_{N,s}}{2}\right).$$
Then from \eqref{eq:M7} we reach that $w_{\bar\lambda+\varepsilon}=0$ and this is a contradiction to \eqref{eq:sup}. Therefore
$$
\bar\lambda=0\,.
$$
{\bf Step 3:} Finally, the symmetry (and monotonicity) in the $x_1$-direction follows as standard repeating the argument in the $(-x_1)$-direction. The radial symmetry result (and the monotonicity of the solution) follows as well performing the Moving Plane Method in any direction $\nu\in \mathbb{S}^{N-1}$.
\end{proof}
In the second part of this section we  study the asymptotic  behaviour of the solutions to \eqref{eq:a3} as stated in Theorem~\ref{t:3}. To do this,  we are going to use a useful  representation result (Lemma \ref{lem:FrLiSe} here below) and then we   reformulate our equation \eqref{eq:a3} to an equivalent singular equation without linear singular term. Towards this, we shall recall a result by Frank, Lieb and Seringer \cite{FLS-1} (in particular equality $(4.3)$ proved in \cite[pag. 935]{FLS-1}).
\begin{lemma}\label{lem:FrLiSe}
Let $0<\ga<(N-2s)/{2}$. If $u\in C^{\infty}_0(\Rn\setminus \{0\})$ and $v(x)=|x|^{\ga}u$, then 
\Bea &&\int_{\Rn}|\xi|^{2s}|\mathcal{F}(u)|^2  \, d\xi \, - \, \big(\Lambda_{N,s}+\Phi_{N,s}(\ga)\big)\int_{\Rn}|x|^{-2s}|u(x)|^2 \, dx \\
&&\qquad=\frac{a_{N,s}}{2}\int_{\Rn}\int_{\Rn}\frac{|v(x)-v(y)|^2}{|x-y|^{N+2s}} \, \frac{dx}{|x|^{\ga}} \, \frac{dy}{|y|^{\ga}},\Eea
where \be\lab{phi-ns}\Phi_{N,s}(\ga):= 2^{2s}\bigg(\frac{\Gamma(\frac{\ga+2s}{2})\Gamma(\frac{N-\ga}{2})}{\Gamma(\frac{N-\ga-2s}{2})\Gamma(\frac{\ga}{2})}
-\frac{\Gamma^2(\frac{N+2s}{4})}{\Gamma^2(\frac{N-2s}{4})}\bigg).\ee
\end{lemma}
Consider, the function
$\Psi_{N,s}:\bigg[0, \frac{N-2s}{2}\bigg]\to[0,\Lambda_{N,s}]$ defined by
\be\lab{6-12-1}\ga\to \Psi_{N,s}(\ga):=\Lambda_{N,s}+\Phi_{N,s}(\ga).\ee Then
$\Psi_{N,s}$ is strictly increasing and surjective (see \cite[Lemma 3.2]{FLS-1}). Therefore, given $\theta\in(0,\Lambda_{N,s})$, there exists unique $\ga\in\big(0, (N-2s)/2\big)$ such that 
\begin{equation}\label{psi-theta-hairagione}
\Psi_{N,s}(\ga)=\theta,
\end{equation}
Substituting $\Lambda_{N,s}$ from \eqref{La-ns} to \eqref{6-12-1} yields
\be\lab{psi-theta}\theta=2^{2s}\frac{\Gamma(\frac{\ga+2s}{2})\Gamma(\frac{N-\ga}{2})}{\Gamma(\frac{N-\ga-2s}{2})\Gamma(\frac{\ga}{2})}.\ee
Let us  define the space $\dot H^{s,\gamma}(\mathbb{R}^N)$
as the closure of $C^{\infty}_0(\mathbb R^N)$ with respect to the norm
$$\| \phi\|_{\dot H^{s,\gamma}(\mathbb{R}^N)}= \left(\int_{\R^N}\frac{|\phi(x)|^{2^*_s}}{|x|^{\gamma 2^*_s}}
\,dx\right)^{\frac{1}{2^*_s}} +
\left(\int_{\mathbb{R}^N}\int_{\mathbb{R}^N}\frac{|\phi(x)-\phi(y)|^2}{|x-y|^{N+2s}}\frac{dx}
{|x|^\gamma}\frac{dy}{|y|^\gamma}\right)^{\frac12}.$$
We also define
$$ \dot W^{s,\gamma}(\R^N):=\{\phi:\R^N\to\R {\mbox{ measurable }}: \| \phi\|_{\dot H^{s,\gamma}(\mathbb{R}^N)}<+\infty\}.$$
Note that by \cite{DV}, one has that
the space $\dot W^{s,\gamma}(\R^N)$ coincides with $\dot H^{s,\gamma}(\R^N)$.

As a consequence of the \textit{ground state representation} given by Lemma~\ref{lem:FrLiSe}, we will transform our problem~\eqref{eq:a3}
into another nonlocal problem in weighted spaces. Namely, we consider~$u\in \dot{H}^s(\mathbb{R}^N)\cap L^{q+1}(\Rn)$  a solution of the problem~\eqref{eq:a3} and set  $v=v_\ga:=|x|^{\ga}u$, with  $\gamma$ given by \eqref{psi-theta-hairagione}.
Furthermore, we introduce the operator~$(-\Delta_{\gamma})^s$, defined as duality product
\begin{equation}\label{operatoralpha}
\big\langle(-\Delta_{\gamma})^s v,\phi\big\rangle=\int_{\mathbb{R}^N}\int_{\mathbb{R}^N}
\frac{(v(x)-v(y))(\phi(x)-\phi(y))}{|x-y|^{N+2s}}\frac{dx}{|x|^\gamma}\frac{dy}{|y|^\gamma},
\end{equation}
for any~$v,\phi\in \dot H^{s,\gamma}(\R^N)$.
Using this representation, if $u$ solves \eqref{eq:a3}, then 
\begin{equation}\label{eq:ladefdiv}
v=|x|^{\ga}u
\end{equation} is a weak solution to 
\begin{equation}
  \label{eq:a3'}
\left\{\begin{aligned}
      (-\De_{\ga})^s v &=\frac{v^p}{|x|^{(p+1)\ga}} - \frac{v^q}{|x|^{(q+1)\ga}} \quad\text{in}\,\, \Rn\\
      v &> 0 \quad\text{in}\,\,\Rn\\
            v &\in \dot{H}^{s,\ga}(\Rn)\cap L^{q+1}(\Rn, |x|^{-(q+1)\ga}), 
 \end{aligned}
  \right.
\end{equation}
We also have (see \cite{AB, DV}) 
\begin{equation}\label{eq:poverino}
C\bigg(\int_{\Rn}\frac{|v(x)|^{2^*_s}}{|x|^{2^*_s\ga}}dx\bigg)^\frac{2}{2^*_s}\leq \|v\|^2_{\dot{H}^{s,\ga}(\Rn)}.
\end{equation}
In the following proposition we prove the boundedness of the weak solutions to  \eqref{eq:a3'}. This result will be the key in the proof of Theorem \ref{t:3}. We have
\begin{proposition}\lab{l:5}
Let $p>2^{*}_s-1$, 
$$q+1>\frac{N}{2s}(p-1)$$
and let $v$ be a solution of  \eqref{eq:a3'}.
Then $v\in L^{\infty}(\mathbb R^N)$.
\end{proposition}
\begin{proof}
We aim to apply Moser iteration technique to prove this theorem. For $\beta\geq 1$ and $T>0$, we define
 \begin{equation} \label{31-10-1}
  \phi(t) := \phi_{\ba, T}(t)=
\left\{\begin{aligned}
      &t^\ba &&\quad\text{if }\quad 0\leq t\leq T, \\
      &\ba T^{\ba-1}(t-T)+T^{\ba}, &&\quad\text{if }\quad t> T.
 \end{aligned}
  \right.
\end{equation}
Since  $v$ satisfies the problem \eqref{eq:a3'} and $\phi_{\ba, T}$ is a Lipschitz function, it follows that $$\phi_{\ba, T}(v)\in \dot{H}^{s,\ga}(\Rn)\cap L^{q+1}(\Rn, |x|^{-(q+1)\ga}).$$ Moreover, in the weak distribution sense, we have  that
\begin{equation}\label{eq:convex}
(-\De_{\ga})^s \phi(v)\leq \phi'(v)(-\De_{\ga})^s v, \quad v\in  \dot{H}^{s,\ga}(\Rn),
\end{equation}
see \cite{DMPS}. 
By using the weighted Sobolev inequality \cite{AB, DV} we have
\begin{eqnarray}\label{eq:left}
&&\frac{a_{N,s}}{2}
\int_{\mathbb R^N} \int_{\mathbb R^N}\frac{|\phi(v(x))-\phi(v(y))|^2}{|x-y|^{N+2s}}\frac{dx}{|x|^{\gamma}}\frac{dy}{|y|^{\gamma}}\\\nonumber
&\geq&\frac{a_{N,s}}{2}S(N,s,\gamma)\left(\int_{\mathbb{R}^N}|\phi(v)|^{2^*_s}\frac{dx}{|x|^{\gamma\cdot 2^*_s}}\right)^{\frac{2}{2^*_s}}.
\end{eqnarray}
On the other hand, using \eqref{eq:convex},  we  get
\begin{eqnarray}\label{eq: right}
&&\int_{\mathbb{R}^N}\phi(v)(-\Delta_\gamma)^s\phi(v)\leq\int_{\mathbb{R}^N}\phi(v)\phi'(v)(-\Delta_\gamma)^sv\\\nonumber&=&\int_{\mathbb{R}^N}\phi(v)\phi'(v)\left(\frac{v^{p}}{|x|^{(p+1)\ga}} - \frac{v^{q}}{|x|^{(q+1)\ga}}  \right) dx\\\nonumber
&&\leq\beta\int_{\mathbb{R}^N} (\phi (v))^2\left(\frac{v^{p-1}}{|x|^{(p+1)\ga}} - \frac{v^{q-1}}{|x|^{(q+1)\ga}}  \right)dx,
\end{eqnarray}
where we used that $t\phi'(t)\leq \beta \phi (t)$. Using  \eqref{eq:left} and \eqref{eq: right} we obtain
\bea\lab{31-10-2}
\left(\int_{\Rn} |\phi(v)|^{2^*_s}\frac{dx}{|x|^{2^{*}_s \ga}} \right)^\frac{2}{2^*_s} &\leq& C\ba\int_{\Rn}\big(\phi(v)\big)^2\left(\frac{v^{p-1}}{|x|^{(p+1)\ga}} - \frac{v^{q-1}}{|x|^{(q+1)\ga}} \right) \,  dx,\no\\
&=& C\ba (A+B),
\eea
where $C=C(s,N,\gamma)$ and
$$A=\int_{B_1}\big(\phi(v)\big)^2\left(\frac{v^{p-1}}{|x|^{(p+1)\ga}} - \frac{v^{q-1}}{|x|^{(q+1)\ga}}  \right) \,  dx$$
and 
\begin{equation}\label{eq:defdiB}B=\int_{B_1^c}\big(\phi(v)\big)^2\left(\frac{v^{p-1}}{|x|^{(p+1)\ga}} - \frac{v^{q-1}}{|x|^{(q+1)\ga}} \right) \,  dx,
\end{equation}
where $B_1^c:=\mathbb R^N\setminus B^1$.

\noindent {\bf Step 1:} We prove that 
\begin{equation}\label{eq:unmedicoinfam}
\int_{\Rn} \frac{v^{2^*_s\ba}}{|x|^{2^{*}_s\ga }}\, dx \leq  C,\quad  \text{for} \,\, \ba=\frac{2^*_s}{2}.
\end{equation}

To prove this step, first we estimate $B$.
\bea\lab{5-7-3}\\\no
&&B\leq \int_{B_1^c \cap\{v\leq m\}}\big(\phi(v)\big)^2\frac{v^{p-1}}{|x|^{(p+1)\ga}} \, {\rm d}x+ \int_{B_1^c \cap\{v\geq m\}}\big(\phi(v)\big)^2\frac{v^{p-1}}{|x|^{(p+1)\ga}} \, {\rm d}x \\\no
&&\leq m^{p-1}\int_{B_1^c \cap\{v\leq m\}}\frac{\big(\phi(v)\big)^2}{|x|^{2^*_s\ga}} \, {\rm d}x+ \int_{B_1^c \cap\{v\geq m\}}\frac{\big(\phi(v)\big)^2}{|x|^{2\ga}}\frac{v^{p-1}}{|x|^{(p-1)\ga}} \, {\rm d}x \\\no
 &&\leq m^{p-1}\int_{B_1^c \cap\{v\leq m\}}\frac{\big(\phi(v)\big)^2}{|x|^{2^*_s\ga}} \, {\rm d}x 
+ \bigg(\int_{B_1^c \cap\{v\geq m\}}\frac{\big(\phi(v)\big)^{2^*_s}}{|x|^{2^*_s\ga}} \, {\rm d}x\bigg)^\frac{2}{2^*_s}\bigg (\int_{B_1^c \cap\{v\geq m\}}\frac{v^{(p-1)(\frac{2^*_s}{2^*_s-2})}}{|x|^{(p-1)\frac{2^*_s\ga}{2^*_s-2}}} \, {\rm d} x\bigg)^\frac{2^*_s-2}{2^*_s}\\\nonumber
&&\leq m^{p-1}\int_{\mathbb R^N}\frac{\big(\phi(v)\big)^2}{|x|^{2^*_s\ga}} \, {\rm d}x 
+ \bigg(\int_{\mathbb R^N}\frac{\big(\phi(v)\big)^{2^*_s}}{|x|^{2^*_s\ga}} \, {\rm d}x\bigg)^\frac{2}{2^*_s}\bigg (\int_{\mathbb R^N \cap\{v\geq m\}}\frac{v^{(p-1)(\frac{2^*_s}{2^*_s-2})}}{|x|^{(p-1)\frac{2^*_s\ga}{2^*_s-2}}} \, {\rm d} x\bigg)^\frac{2^*_s-2}{2^*_s}.
\eea 
We observe that under our assumptions we have 
$$ 2^*_s< (p-1)\frac{2^*_s}{2^*_s-2}< (q+1).$$ Therefore by interpolation inequality we infer that  
\be\lab{5-7-2}\frac{v^{(p-1)(\frac{2^*_s}{2^*_s-2})}}{|x|^{(p-1)\frac{2^*_s\ga}{2^*_s-2}}}\in L^1(\Rn).\ee
Next, we estimate $A$ as below:
\be\lab{5-7-1}A=A_1-A_2,\ee
where
$$A_1:=\int_{B_1}\big(\phi(v)\big)^2\frac{v^{p-1}}{|x|^{(p+1)\ga}} \, dx\quad\text{and}\quad A_2:=\int_{B_1}\big(\phi(v)\big)^2\frac{v^{q-1}}{|x|^{(q+1)\ga}}\,  dx.$$
We also observe that using weighted Young's inequality we get 
\bea
\frac{v^{p-1}}{|x|^{(p+1)\ga}}&=&\frac{v^{p-1}}{|x|^{\big(2^*_s(\frac{p-1}{q-1})+(p+1)-2^*_s\big)\ga}}\cdot \frac{1}{|x|^{2^*_s\ga(\frac{q-p}{q-1})}}\no\\
&\leq&\frac{\varepsilon v^{q-1}}{|x|^{\big(2^*_s+(\frac{q-1}{p-1})((p+1)-2^*_s)\big)\ga}}+\frac{C(p,q,\eps)}{|x|^{2^*_s\ga}}.\no
\eea
Therefore
\bea
A_1&\leq&\varepsilon\int_{B_1}\frac{(\phi(v))^2v^{q-1}}{|x|^{(2^*_s+(\frac{q-1}{p-1})((p+1)-2^*_s))\ga}} \, dx+C(p,q,\eps)\int_{B_1}\frac{\big(\phi(v)\big)^2}{|x|^{2^*_s\ga}} \, dx.\no
\eea
Note that since $q>p$
$$2^*_s+\bigg(\frac{q-1}{p-1}\bigg)\big((p+1)-2^*_s\big)< q+1$$
and therefore
$$A_1\leq \varepsilon \int_{B_1}\frac{\big(\phi(v)\big)^2v^{q-1}}{|x|^{(q+1)\ga}} \, dx+C(p,q,\eps)\int_{B_1}\frac{\big(\phi(v)\big)^2}{|x|^{2^*_s\ga}} \, dx.$$
Thus, choosing 
$\varepsilon<<1$, from \eqref{5-7-1} we obtain
\be\lab{5-7-4}A\leq C(p,q,\eps)\int_{B_1}\frac{\big(\phi(v)\big)^2}{|x|^{2^*_s\ga}} \, dx\leq C\int_{\Rn}\frac{\big(\phi(v)\big)^2}{|x|^{2^*_s\ga}} \, dx.\ee
Putting together \eqref{5-7-3} and \eqref{5-7-4} into \eqref{31-10-2}, we have
\begin{eqnarray}\label{eq:sntris}\\\nonumber
&&\left(\int_{\Rn} \frac{|\phi(v)|^{2^*_s}}{|x|^{2^{*}_s}\ga } \, dx\right)^\frac{2}{2^*_s} \\\nonumber&&\leq  C\ba\left( (1+m^{p-1})\int_{\Rn}\frac{\big(\phi(v)\big)^2}{|x|^{2^*_s\ga}} \, dx +\left(\int_{\Rn}\frac{\big(\phi(v)\big)^{2^*_s}}{|x|^{2^*_s\ga}} \, dx\right)^\frac{2}{2^*_s}
 \left (\int_{\mathbb R^N \cap\{v\geq m\}}\frac{v^{(p-1)(\frac{2^*_s}{2^*_s-2})}}{|x|^{(p-1)\frac{2^*_s\ga}{2^*_s-2}}} \, dx\right)^\frac{2^*_s-2}{2^*_s} \right),
\end{eqnarray}
with  $C=C(p,q,s,N,\gamma)$ a positive constant. 
By \eqref{5-7-2}, we can fix $m$ such that 
$$
\left (\int_{\mathbb R^N \cap\{v\geq m\}}\frac{v^{(p-1)(\frac{2^*_s}{2^*_s-2})}}{|x|^{(p-1)\frac{2^*_s\ga}{2^*_s-2}}} \, dx\right)^\frac{2^*_s-2}{2^*_s}\leq \frac{1}{2C\beta},$$
where $C$ is given in \eqref{eq:sntris}. Hence
\begin{align*}
\left(\int_{\Rn} \frac{|\phi(v)|^{2^*_s}}{|x|^{2^{*}_s\ga}} \, dx\right)^\frac{2}{2^*_s} &\leq  C\ba \int_{\Rn}\frac{\big(\phi(v)\big)^2}{|x|^{2^*_s\ga}} \, dx \\\nonumber
&\leq C\ba \int_{\Rn}\frac{v^{2\beta}}{|x|^{2^*_s\ga}} \, dx<+\infty,
\end{align*}
up to redefine the constant $C$ and where we used that $\phi(t)\leq t^{\ba}$, $\beta=2^*_s/2$ and
$$\int_{\Rn}\frac{ v^{2^*_s}}{|x|^{2^{*}_s\ga }}\, dx< +\infty,$$
since \eqref{eq:poverino} holds. By Fatou's lemma, taking  $T\to\infty$, we obtain
\be\lab{5-7-6}
\left (\int_{\Rn} \frac{v^{2^*_s\ba}}{|x|^{2^{*}_s\ga }} \, dx\right)^\frac{2}{2^*_s} \leq  C.
\ee
{\bf Step 2}: In this step we establish the iteration formula, i.e. the equation \eqref{5-7-13} below. 

\noindent Towards this, first we observe that from the definition of $B$ \eqref{eq:defdiB}, we can also write
$$B\leq \int_{B_1^c}\big(\phi(v)\big)^2\frac{v^{p-1}}{|x|^{(p+1)\ga}} \, dx.$$ 
Since $p+1>2^*_s$, using \eqref{eq:ladefdiv} and Lemma \ref{l:decay} (see in particular \eqref{eq:vascovita}) we get the following estimate
\begin{align}\lab{5-7-7}
B&\leq \int_{B_1^c}\big(\phi(v)\big)^2\frac{v^{p-1}}{|x|^{(p+1)\ga}} \, dx=\int_{B_1^c}\big(\phi(v)\big)^2\frac{v^{2^*_s-2+p+1-2^*_s}}{|x|^{(p+1)\ga}} \, dx\\\nonumber
&=\int_{B_1^c}\big(\phi(v)\big)^2\frac{v^{2^*_s-2}u^{p+1-2^*_s}}{|x|^{2^*_s\ga}} \, dx\leq C(p,q,s,N,\|u\|_{L^{q+1}(\Rn)})\int_{B_1^c}\big(\phi(v)\big)^2\frac{v^{2^*_s-2}}{|x|^{2^*_s\ga}} \, dx\\\nonumber
&\leq C(p,q,s,N,\|u\|_{L^{q+1}(\Rn)})\int_{\mathbb R^N}\big(\phi(v)\big)^2\frac{v^{2^*_s-2}}{|x|^{2^*_s\ga}} \, dx
\end{align}
On the other hand, from the first inequality of \eqref{5-7-4}, we can estimate $A$ further as follows:
$$A\leq C(p,q)\left(\int_{B_1\cap \{v\geq 1/2\}}\frac{\big(\phi(v)\big)^2}{|x|^{2^*_s\ga}} \, dx \, + \,
\int_{B_1\cap \{v\leq 1/2\}}\frac{\big(\phi(v)\big)^2}{|x|^{2^*_s\ga}} \, dx\right).$$
Therefore \big(recall $\phi(t)\leq t^{\ba}$ and that $0<\ga<(N-2s)/{2}$\big)
\bea\lab{5-7-8}
A&\leq& {C(p,q,s,N)}\int_{B_1\cap \{v\geq 1/2\}}\frac{\big(\phi(v)\big)^2v^{2^*_s-2 }}{|x|^{2^*_s\ga}} \, dx + C(p,q)4^{-\ba}\int_{B_1}\frac{dx}{|x|^{2^*_s\ga}}\no\\
&\leq& C(p,q ,s,N,\gamma)\left(\int_{\Rn} \frac{\big(\phi(v)\big)^2v^{2^*_s-2 }}{|x|^{2^*_s\ga}} \, dx +1\right).
\eea
Combining \eqref{5-7-7} and \eqref{5-7-8} along with \eqref{31-10-2}, we have
\begin{align}\label{eq:frankesp}
\left(\int_{\Rn} \frac{|\phi(v)|^{2^*_s}}{|x|^{2^{*}_s\ga}} \, dx\right)^\frac{2}{2^*_s} &\leq C\ba\Bigg(C(p,q,s,N,\gamma)
\left(\int_{\Rn} \frac{\big(\phi(v)\big)^2v^{2^*_s-2 }}{|x|^{2^*_s\ga}} \, dx +1\right)\\\nonumber
& \quad+C(p,q,s,N,\|u\|_{L^{q+1}(\mathbb R^N)})\int_{\mathbb R^N}\big(\phi(v)\big)^2\frac{v^{2^*_s-2}}{|x|^{2^*_s\ga}} \, dx\Bigg)
\\\nonumber &\leq C\ba\left(\int_{\Rn} \frac{v^{2\ba+2^*_s-2}}{|x|^{2^*_s\ga}} \, dx +1\right),\no
\end{align}
where $C=C(p,q,s,N,\gamma,\|u\|_{L^{q+1}(\mathbb R^N)})$ is a suitable positive constant.
From the definition of $\phi(t)$, we know that $\phi(t)= t^{\ba}$ if $t\leq T$. Using this and Fatou's lemma, we obtain 
\begin{equation*}\left(\int_{\Rn} \frac{v^{2^*_s\ba}}{|x|^{2^{*}_s\ga }} \, dx\right)^\frac{2}{2^*_s} \leq C\ba\left(\int_{\Rn} \frac{v^{2\ba+2^*_s-2}}{|x|^{2^*_s\ga}} \, dx +1\right).
\end{equation*}
Hence
$$\left(\int_{\Rn} \frac{v^{2^*_s\ba}}{|x|^{2^{*}_s\ga }} \, dx+1\right)^\frac{2}{2^*_s}\leq \left(\int_{\Rn} \frac{v^{2^*_s\ba}}{|x|^{2^{*}_s\ga}} \, dx\right)^\frac{2}{2^*_s}+1\leq C\ba\left(\int_{\Rn} \frac{v^{2\ba+2^*_s-2}}{|x|^{2^*_s\ga}} \, dx +1\right),$$
up to redefine the constant $C$ given in \eqref{eq:frankesp}. This implies
\be\lab{5-7-9}\left(\int_{\Rn} \frac{v^{2^*_s\ba}}{|x|^{2^{*}_s\ga }} \, dx+1\right)^\frac{1}{2^*_s(\ba-1)}\leq (C\ba)^\frac{1}{2(\ba-1)}\left(\int_{\Rn} \frac{v^{2\ba+2^*_s-2}}{|x|^{2^*_s\ga}} \, dx +1\right)^\frac{1}{2(\beta-1)}.\ee
For $k\geq 1$, let  us define $\{\beta_{k}\}$ by
\begin{equation}\lab{5-7-10}2 \beta_{k+1}+2^*_s-2=2^*_s \beta_k \quad \text{and} \quad \beta_1=\frac{2^*_s}{2}.
\end{equation} 
Thus taking $\ba=\ba_{k+1}$ in \eqref{5-7-9}, we get
\be\lab{5-7-11}\left(\int_{\Rn} \frac{v^{2^*_s\ba_{k+1}}}{|x|^{2^{*}_s\ga}} \, dx+1\right)^\frac{1}{2^*_s(\ba_{k+1}-1)}\leq (C\ba_{k+1})^\frac{1}{2(\ba_{k+1}-1)}\bigg(\int_{\Rn}\frac{v^{2^*_s\ba_k}}{|x|^{2^*_s\ga}} \, dx +1\bigg)^\frac{1}{2(\ba_{k+1}-1)}.\ee
Moreover,  \eqref{5-7-10}  implies
\be\lab{5-7-12}
2(\ba_{k+1}-1)=2^*_s(\ba_k-1).
\ee
Consequently, we can rewrite \eqref{5-7-11} as
\be\lab{5-7-13}\left( 1+\int_{\Rn} \frac{v^{2^*_s\ba_{k+1}}}{|x|^{2^{*}_s\ga }} \, dx\right)^\frac{1}{2^*_s(\ba_{k+1}-1)}\leq (C\ba_{k+1})^\frac{1}{2(\ba_{k+1}-1)}\left(1+\int_{\Rn}\frac{v^{2^*_s\ba_k}}{|x|^{2^*_s\ga}} \, dx \right)^\frac{1}{2^*_s(\ba_k-1)}.\ee
Iterating the relation \eqref{5-7-12} yields
\be\lab{5-7-14} 
\ba_{k+1}-1=\bigg(\frac{2^*_s}{2}\bigg)^k(\ba_1-1).
\ee
Therefore, $\ba_{k+1}\to\infty$. If we denote
$$A_k:=\left( 1+\int_{\Rn} \frac{v^{2^*_s\ba_{k}}}{|x|^{2^{*}_s\ga }} \, dx\right)^\frac{1}{2^*_s(\ba_{k}-1)}, \quad C_k=(C\ba_{k})^\frac{1}{2(\ba_{k}-1)}$$
we get the recurrence formula $A_{k+1}\leq C_{k+1}A_k$, $k\geq1$. 

\noindent {\bf Step 3}: The conclusion, namely we prove that $\|v\|_{L^{\infty}(\mathbb{R}^N)}\leq C$.

\noindent Arguing by induction  we have
 \begin{eqnarray}\label{eq:induction}
 \log A_{k+1}&\leq &\sum_{j=2}^{k+1}\log C_{j}+\log A_1\\\nonumber&\leq&  \sum_{j=2}^{+\infty}\log C_{j}+\log A_1 < +\infty.
 \end{eqnarray}
 Indeed the series $\displaystyle \sum_{j=2}^{+\infty}\log C_{j}< +\infty$ is convergent: think that $\beta_{k+1}=\beta_1^{k}(\beta_1-1)+1$, see \eqref{5-7-14}   and $A_1\leq C$, see \eqref{eq:unmedicoinfam}.
Using \eqref{eq:induction}, it follows
$$\log\left (\left(\int_{\Rn} \frac{v^{2^*_s\ba_{k+1}}}{|x|^{2^{*}_s\ga }} \, dx\right)^\frac{1}{2^*_s(\ba_{k+1}-1)}\right)\leq\log\left (\left( 1+\int_{\Rn} \frac{v^{2^*_s\ba_{k+1}}}{|x|^{2^{*}_s\ga }} \, dx\right)^\frac{1}{2^*_s(\ba_{k+1}-1)}\right)\leq C$$ and then, for $R>0$ fixed
$$\frac{\gamma}{(\beta_{k+1}-1)}\log\frac 1R +\log\left(
\left( \int_{B_R}v^{\beta_{k+1} \cdot 2^*_s}\,{dx} \right)^{\frac{1}{2^*_s(\beta_{k+1}-1)}}  \right)\leq C.$$
Since (see \eqref{5-7-14}) $\beta_{k}\rightarrow +\infty$ as $k\rightarrow  \infty$, we have
$$\log \left(\left(\int_{B_R}|v|^{\beta_{k+1} \cdot 2^*_s}\,{dx}\right)^{\frac{1}{2^*_s(\beta_{k+1}-1)}}\right)\leq C,$$
with $C$ a positive constant not depending on $R$.
This end the proof since
$$\lim_{k\rightarrow +\infty}\left(\int_{B_R}|v|^{\beta_{k+1} \cdot 2^*_s}\,{dx}\right)^{\frac{1}{2^*_s\beta_{k+1}}}=\|v\|_{L^{\infty}(B_R)}$$
and thus
$$\|v\|_{L^{\infty}(\mathbb{R}^N)}\leq C,$$
where $C=C(p,q,s,N,\gamma,\|u\|_{L^{q+1}(\mathbb R^N)})$. This  end the proof.
\end{proof}
\begin{proof}[\bf {Proof of Theorem~\ref{t:3}}] Let $\gamma_{\theta}$  the unique solution of \eqref{psi-theta-hairagione}.  Setting  $v_{\gamma_{\theta}}:=|x|^{\gamma_{\theta}}u$, by    Proposition~\eqref{l:5} we have that 
$$v_{\gamma_{\theta}}(x)\leq C,\quad x\in \mathbb R^N,$$
for some positive constant $C=C(p,q,s,N,\gamma_{\theta},\|u\|_{L^{q+1}(\mathbb R^N)})$.
Then
$$0<u(x)\leq C |x|^{-\gamma_{\theta}}, \quad |x|>0.$$
\end{proof}

\vspace{3mm}

\noindent {\bf Acknowledgement}:  The research of M.~Bhakta is partially supported by the {\em SERB MATRICS grant (MTR/2017/000168) and WEA grant (WEA/2020/000005)}. D.~Ganguly is partially supported by INSPIRE faculty fellowship (IFA17-MA98). L. Montoro is partially supported by PRIN project  2017JPCAPN (Italy): {\em Qualitative and quantitative aspects of nonlinear PDEs} and Project  PDI2019-110712GB-100, 2020-2023, MICINN  (Spain): {\em Ecuaciones con perturbaciones de potencias del Laplaciano}.

\end{document}